\documentclass[reqno,11pt]{amsart}

\usepackage[french,english]{babel}
\usepackage[colorlinks=true, linkcolor=blue, urlcolor=blue]{hyperref}

\usepackage[utf8]{inputenc}

\usepackage[T1]{fontenc}
\usepackage{pgfplots}
\pgfplotsset{compat=1.17}
\usepackage{amsmath}
\usepackage{amsfonts}
\numberwithin{equation}{section} 
\usepackage{comment}
\usepackage{amssymb}
\usepackage{lmodern}
\usepackage{amsmath}
\usepackage{amssymb}
\usepackage{mathrsfs}
\usepackage{amsthm}
\usepackage{enumerate}
\usepackage{dsfont} % pour la fonction caractéristique, \mathds{1}
\usepackage{eurosym} % pour le signe  : \euro{}
\usepackage{stmaryrd} % pour les intervalles entiers : $\llbracket 1;n\rrbracket$ pour [|1;n|]

\usepackage{tikz}
\usetikzlibrary{arrows, decorations.pathreplacing}

\usepackage{graphicx}
\usepackage{float}

\usepackage{hyperref}  % pour les hyperliens (table des matières notamment)
\hypersetup{                    % parametrage des hyperliens
	colorlinks=true,                % colorise les liens
	breaklinks=true,                % permet les retours à la ligne pour les liens trop longs
	urlcolor= blue,                 % couleur des hyperliens
	linkcolor= black,                % couleur des liens internes aux documents (index, figures, tableaux, equations,...)
	citecolor= magenta,                % couleur des liens vers les references bibliographiques
	pdfstartview = FitH,
	bookmarksopen = true               % marques-page déroulés
}

\usepackage{listings}
\usepackage{xcolor}
\usepackage{framed}

\usepackage{fullpage} % pour utiliser toute la page

\theoremstyle{definition}

\theoremstyle{plain}
\newtheorem{demo}{Démonstration}

\newcommand{\R}{\mathbb{R}}
\newcommand{\N}{\mathbb{N}}
\newcommand{\E}{\mathbb{E}}

\title{On the surjectivity of the conditional expectation\\ given a real random variable}
\date{\today}

\numberwithin{equation}{section} % numéroter les équations par section

\usepackage{amsthm,amsmath}
\numberwithin{equation}{section}

\newtheorem{thm2}{Theorem}[section]
\newtheorem{lemma}[thm2]{Lemma}
\newtheorem{cor2}[thm2]{Corollary}
\newtheorem{prop2}[thm2]{Proposition}

\theoremstyle{definition}
\newtheorem{def2}[thm2]{Definition}
\newtheorem{example}[thm2]{Example}
\newtheorem{rk}[thm2]{Remark}
\addto\captionsfrench{}

\allowdisplaybreaks

\usepackage{todonotes}
\newcommand{\julien}[1]{\todo[inline, backgroundcolor=red!20!white]{JG: #1}}

\author{
  Julien Guyon$^{1,2}$, 
  Thibault Jeannin$^{1}$, 
  Benjamin Jourdain$^{1}$}

\date{}

\usepackage[utf8]{inputenc}
\usepackage{fancyhdr}
\begin{document}

\maketitle

\renewcommand{\thefootnote}{\arabic{footnote}} 

\footnotetext[1]{CERMICS, ENPC, Institut Polytechnique de Paris, Inria, Marne-la-Vallée, France. Emails: 
\href{mailto:julien.guyon@enpc.fr}{julien.guyon@enpc.fr}, 
\href{mailto:thibault.jeannin@enpc.fr}{thibault.jeannin@enpc.fr}, 
\href{mailto:benjamin.jourdain@enpc.fr}{benjamin.jourdain@enpc.fr}.}

\footnotetext[2]{NYU Tandon School of Engineering, Department of Finance and Risk Engineering, One MetroTech
Center, Brooklyn, NY 11201, USA. Email: 
\href{mailto:julien.guyon@nyu.edu}{julien.guyon@nyu.edu}.\\ This research benefited from the support of
the chair {\it Risques Financiers}, Fondation du Risque, and from the BNP Paribas chair {\it Futures of Quantitative Finance}.}

\makeatletter
\let\@fnsymbol\@arabic % pour ne pas avoir de *
\makeatother

\renewcommand{\thefootnote}{} % supprime le numéro
\begin{comment}
\footnotetext{
The first author acknowledges financial support from the BNP Paribas Chair “Futures of Quantitative Finance.” The second author acknowledges financial support from the "MATHRISK" Chair, and from Bloomberg through the Quant Finance Fellowship as part of his PhD.
}
\end{comment}
\addtocounter{footnote}{-1} % évite que le prochain footnote soit numéroté 2 au lieu de 1
\renewcommand{\thefootnote}{\arabic{footnote}} % on remet la numérotation normale

 \begin{abstract}
In this paper, we investigate the distributions of random couples $(X,Y)$ with $X$ real-valued such that any non-negative integrable random variable $f(X)$ can be represented as a conditional expectation, $f(X)=\E[g(Y)|X]$, for some non-negative measurable function $g$. It turns out that this representation property is related to the smallness of the support of the conditional law of $X$ given $Y$, and in particular fails when this conditional law almost surely has a non-zero absolutely continuous component with respect to the Lebesgue measure. We give a sufficient condition for the representation property and check that it is also necessary under some additional assumptions (for instance when $X$ or $Y$ are discrete). We also exhibit a rather involved example where the representation property holds but the sufficient condition does not. Finally, we discuss a weakened representation property where the non-negativity of $g$ is relaxed. This study is motivated by the calibration of time-discretized path-dependent volatility models to the implied volatility surface.
\end{abstract}

\maketitle


\maketitle

\section{Introduction}

Recent studies in quantitative finance \cite{ Zumbach2009,
Zumbach2010,  ChicheporticheBouchaud2014, Bouchaud2021,  GuyonLekeufack2023, AndreBoumezoudJourdain2023} have highlighted the significant path-dependence of volatility: it is very well explained by past asset returns—but not by just the current asset price. Those studies show that the instantaneous volatility of financial markets is mostly endogenous and that it depends only marginally on extra random factors. For example, in~\cite{GuyonLekeufack2023}, volatility is modeled as a function of a path-dependent variable
\[
Y_t = (R_{1,t}, R_{2,t}),
\]
where \( R_{1,t} \) and \( R_{2,t} \) are, respectively, convex combinations of exponentially weighted averages of past returns and past squared returns. If \( X_t \) denotes the price of the underlying asset, the instantaneous volatility is expressed as
\begin{equation} \label{eq:vol_model}
\sigma_t = \beta_0 + \beta_1 R_{1,t} + \beta_2 \sqrt{R_{2,t}},
\end{equation}
with
\[
R_{1,t} = (1-\theta_1)\int_{-\infty}^t \lambda_{1,0}\, e^{-\lambda_{1,0} (t - s)} \, \frac{dX_s}{X_s}+\theta_1 \int_{-\infty}^t \lambda_{1,1}\, e^{-\lambda_{1,1} (t - s)} \, \frac{dX_s}{X_s}, \quad
\]
\[
R_{2,t} = (1-\theta_2)\int_{-\infty}^t \lambda_{2,0}\, e^{-\lambda_{2,0} (t - s)} \left( \frac{dX_s}{X_s} \right)^2+\theta_2\int_{-\infty}^t \lambda_{2,1} e^{-\lambda_{2,1} (t - s)} \left( \frac{dX_s}{X_s} \right)^2.
\]
Path-dependent volatility (PDV) models are a new breed of models that precisely capture this endogeneity of markets, and PDV is a new, very promising paradigm of volatility modeling in finance. A natural question then arises: is it possible to build a ``pure'' PDV model\footnote[3]{By pure PDV model, we mean that in \eqref{eq:model}, $\sigma(t,Y_t)$ is not multiplied by a leverage function $\ell(t,X_t)$ like in \eqref{eq:smileSDE}.} that exactly fits an arbitrage-free implied volatility surface (IVS)? More precisely, once relevant path-dependent features \( Y_t \) have been chosen, assuming we observe an IVS in a market with zero rates, repo and dividends, is it possible to build a deterministic function \( \sigma \) such that the following model, where $(W_t)_{t\ge 0}$ is a Brownian motion,
\begin{equation}
    \frac{dX_t}{X_t} = \sigma(t, Y_t) \,dW_t,
    \label{eq:model}
\end{equation}
is exactly calibrated to the IVS observed in the market? This would reconcile the path-dependency of volatility with the path-independence of vanilla option payoffs. We recall that according to Gyöngy \cite{Gyongy1986} and Dupire \cite{Dupire1994}, a necessary and sufficient condition for model \eqref{eq:model} to exactly fit the IVS is
\begin{equation}
    \forall t \geq 0,\quad \,\mathbb{E}[\sigma^2(t, Y_t) | X_t] = \sigma^2_{\text{loc}}(t, X_t),
    \label{eq:calibration}
  \end{equation}
  where the local volatility function is given from the market prices $(\mathcal{C}(t,x))_{t,x>0}$ of call options with maturity $t$ and strike $x$ by Dupire's formula 
  $$\sigma^2_{\text{loc}}(t, x)=\frac{2\,\partial_t\mathcal{C}(t,x)}{x^2\,\partial_{xx}^2\mathcal{C}(t,x)}.$$
  The implicit and complicated dependence of the distribution of $(X_t,Y_t)_{t\ge 0}$ on $\sigma$ makes the existence of $\sigma$ such that \eqref{eq:calibration} holds a very difficult mathematical question.

 A related problem, which shares the same implicit nature, is the calibration to the market IVS of stochastic volatility models $$\frac{dX_t}{X_t}=\varphi(Z_t)\,dW_t,$$
  in which $(Z_t)_{t\ge 0}$ solves an autonomous stochastic differential equation driven by a Brownian motion correlated with $(W_t)_{t\ge 0}$. In such models, the few parameters governing the dynamics of $(Z_t)_{t\ge 0}$ and the function $\varphi$ (there are $4$ parameters in the celebrated Heston model where $(Z_t)_{t\ge 0}$ is a Cox-Ingersoll-Ross process and $\varphi(z)=\sqrt{z}$) do not offer enough flexibility to satisfy \eqref{eq:calibration} and thus perfectly fit to the market IVS.  Instead of considering a non-parametric search for a function $\sigma$ such that $\frac{dX_t}{X_t}=\sigma(t,Z_t)\,dW_t$ is calibrated to the IVS, the efforts to overcome this shortfall and introduce more flexibility have focused on so-called stochastic local volatity models. In these models, the calibration to the IVS is enabled by the choice of a leverage function $\ell$ which multiplies $\varphi(Z_t)$ in the dynamics of the asset price:
  \begin{equation}
\frac{dX_t}{X_t} = \varphi(Z_t)\,\ell(t, X_t)\,dW_t.
\label{eq:smileSDE}
\end{equation}
  The necessary and sufficient condition for the exact fit of the IVS analogous to \eqref{eq:calibration} writes \begin{equation}
\ell^2(t, X_t) = \frac{\sigma^2_{\text{loc}}(t, X_t)}{\mathbb{E}[\varphi(Z_t)^2|X_t]}.\label{eq:effectiveVol}
\end{equation}
 Again, the existence of a leverage function $\ell$ such that this condition holds is a very difficult mathematical question because of the implicit and complicated dependence of the joint law of $(X_t,Z_t)_{t\ge 0}$ on $\ell$ (even if the law of $(Z_t)_{t\ge 0}$ is now fixed). Several articles have discussed the existence of a process \( (X_t, Z_t)_{t \geq 0} \) that satisfies equations \eqref{eq:smileSDE} and \eqref{eq:effectiveVol}. Abergel and Tachet \cite{AbergelTachet} address the existence problem by analyzing a discretized version (in both time and space) of the Fokker–Planck equation associated with equation (1.1). Their strategy is based on a fixed-point argument, which yields a short-time existence result. Similarly, Jourdain and Zhou \cite{JourdainZhou2002} use a Galerkin discretization to examine the question of existence, under the assumption that the process \( (Z_t)_{t \geq 0} \) takes values in a finite set. In contrast, Lacker, Shkolnikov, and Zhang \cite{LackerShkolnikovZhang} consider a setting with continuous time and space, under the restrictive assumption that all coefficients are time-homogeneous, that is,
\[
\ell(t, X_t) = \ell(X_t), \quad \sigma_{\text{loc}}(t, X_t) = \sigma_{\text{loc}}(X_t), \quad
dZ_t = \alpha(Z_t)\,dt + \beta(Z_t)\,dB_t,
\]
where \( W \) and \( B \) are independent. Under a mean-reverting-type assumption on \( (X_t)_{t \geq 0} \), as well as additional regularity assumptions, they establish the existence of a stationary solution. Finally, Djete recently provided in \cite{DjeteNonRegular} a theorem for continuous time and space, possibly non-homogeneous coefficients, correlated Brownian drivers, arbitrary maturities \( T \), and any sufficiently smooth initial condition. However, he only considers instantaneous variances of the form
$$
c + p_X(t, X_t) \,v(t, X_t, Y_t),
$$
where \( c > 0 \), \( p_X(t, \cdot) \) is the density of the law of \( X_t \), and \( Y_t \) is an Itô process.

When going to the Euler discretization with step $h>0$, we get
  $$X^{h}_{(k+1)h}=X^h_{kh}+\ell(kh,X^h_{kh})\,\varphi(Z^h_{kh})\,(W_{(k+1)h}-W_{kh}),\quad k\in\N,$$
  and the law of $(X^h_{kh},Z^h_{kh})$ only depends on $(\ell(jh,\cdot))_{0\le j\le k-1}$ but not on $\ell(kh,\cdot)$ so that the equation
  $$\ell^2(kh,x)=\frac{\sigma^2_{\text{loc}}(kh, x)}{\E[\varphi(Z^h_{kh})^2|X^h_{kh}=x]},\quad x\in\R,$$
  is explicit and can be used to compute $\ell(kh,\cdot)$ recursively as proposed by Guyon and Henry-Labordère \cite{GuyonLabordere} who rely on interacting particles to approximate the conditional expectation.  This numerical method provides an efficient calibration procedure, even if some calibration errors may appear  for large maturities, when the range of the stochastic volatility process $(\varphi(Z_t))_{t\geq 0}$ is very large. Note that according to \cite{FJWZ}, for $T,K>0$, $\E[(X^{T/n}_T-K)^+]-\mathcal{C}(T,K)$ is at most of order $n^{-1/2}$.

  The Euler discretization of the path-dependent volatility model \eqref{eq:model} writes 
  $$X^{h}_{(k+1)h}=X^h_{kh}+\sigma(kh,Y^h_{kh})\,(W_{(k+1)h}-W_{kh}),\quad k\in\N,$$
  and the law of $(X^h_{kh},Z^h_{kh})$ only depends on $(\sigma(jh,\cdot))_{0\le j\le k-1}$ but not on $\sigma(kh,\cdot)$. Nevertheless, the existence of $\sigma^2(kh,\cdot)$ solving the discretized version of \eqref{eq:calibration}, 
\begin{equation}\E[\sigma^2(kh,Y^h_{kh})|X^h_{kh}]=\sigma^2_{\text{loc}}(kh, X^h_{kh}), \label{eq:purePDVsmilecalib_discrete}
\end{equation}
is not obvious even if this equation is linear in unknown $\sigma^2(kh,\cdot)$. Note that no special structure is expected for the local variance function $\sigma^2_{\text{loc}}(kh,\cdot)$ which is an input of the equation. In particular, it may take the value $0$.
  The aim of the present paper is to address this existence issue which we formalize in the following way: under which conditions on the distribution $\pi$ of $(X,Y)$ with $X$ a real-valued random variable and $Y$ with values in ${\mathcal Y}$, do we have
\begin{equation}
\forall f:\mathbb{R} \to \mathbb{R}_+ \text{ such that } \mathbb{E}[f(X)] < \infty,\;
\exists g:\mathcal{Y} \to \mathbb{R}_+ \text{ measurable such that }
\mathbb{E}[g(Y) | X] = f(X) \text{ a.s.?}
\end{equation}
Moreover, from the calibration perspective, it is natural to suppose that $Y$ is finite-dimensional, no further difficulty arises when weakening this condition by assuming that ${\mathcal Y}$ is a separable Banach space.
Then, we denote by $\mathcal{B}({\mathcal Y})$ the Borel sigma algebra on ${\mathcal Y}$. Observe that since both $\mathcal{Y}$ and $\mathbb{R}$ are separable metric spaces, we have (see for example Lemma 6.4.2 of \cite{Bogachev2007v2})
$
\mathcal{B}(\mathbb{R} \times \mathcal{Y}) = \mathcal{B}(\mathbb{R}) \times \mathcal{B}(\mathcal{Y}).
$
Let $\pi$ be a probability measure on $(\mathbb{R} \times \mathcal{Y}, \mathcal{B}(\mathbb{R} \times \mathcal{Y})) = (\mathbb{R} \times \mathcal{Y}, \mathcal{B}(\mathbb{R}) \times \mathcal{B}(\mathcal{Y}))$, with respective marginals $\mu$ and $\nu$ defined by
\[
\mu(A) = \pi(A \times \mathcal{Y})  \quad \text{and} \quad \nu(B) = \pi(\mathbb{R} \times B)
\]
for $A \in \mathcal{B}(\mathbb{R})$ and $B \in \mathcal{B}(\mathcal{Y})$.

By the disintegration theorem (see Theorem 10.22 of \cite{Dudley2002}), there exist probability kernels $(x, B) \mapsto \pi_{X = x}(B)$ on $\mathbb{R} \times \mathcal{B}(\mathcal{Y})$ and $(y, A) \mapsto \pi_{Y = y}(A)$ on $\mathcal{Y} \times \mathcal{B}(\mathbb{R})$ such that
\[
\pi(dx, dy) = \mu(dx)\, \pi_{X = x}(dy) = \nu(dy)\, \pi_{Y = y}(dx),
\]
and $\mu(dx)$-a.e., $\pi_{X = x}$ is unique, while $\nu(dy)$-a.e., $\pi_{Y = y}$ is unique.

Our aim is thus to characterize the probability measures $\pi$ on $\mathbb{R} \times \mathcal{Y}$ such that
\begin{equation}
    \forall f \in L^1_+(\mu),\; \exists g:\mathcal{Y} \to \mathbb{R}_+ \text{ measurable such that } \mu(dx)\text{-a.e.,} \quad
    f(x) = \int_{\mathcal{Y}} g(y)\, \pi_{X = x}(dy).
    \tag{$\mathcal{R}_+$} \label{egtoutf}
\end{equation}
Note that necessarily $g \in L^{1}_{+}(\nu) \text{ since } \|g\|_{L^{1}(\nu)} \;=\; \|f\|_{L^{1}(\mu)}.$ This inverse problem of conditional expectations is general enough to be interesting in its own right, beyond the financial context. To the best of our knowledge, it has never been addressed in the literature.  The fact that the conditioning random variable $X$ is real-valued plays an important role in most of our analysis, but not in the finite-support case that we now consider.
% For \( (X,Y) \) a random vector distributed according to \( \pi \), this surjectivity property of the conditional expectation operator \( Tg(x)=\int_{y\in {\mathcal Y}} g(y)\pi_{X=x}(dy) \) can also be written as:
% \[
% \forall f\in L^1_+(\mu),\;\exists g:\mathbb{R}\to\mathbb{R}_+\text{ such that } \mathbb{E}[g(Y)|X] = f(X) \text{ a.s.}
% \]

\medskip
Let us first suppose that $\mu$ and $\nu$ are finitely supported, i.e., $\mu=\sum_{i=1}^I p_i\,\delta_{x_i}$ and $\nu=\sum_{j=1}^Jq_j\,\delta_{y_j}$ where $I,J\in\N^{*}=\mathbb{N} \setminus \{0\}$, $(p_1,\ldots,p_I)\in(0,1]^I$ and $(q_1,\ldots,q_J)\in(0,1]^J$ are such that $\sum_{i=1}^I p_i=1=\sum_{j=1}^J q_j$, $x_1,\dots,x_I$ and $y_1,\dots,y_J$ are distinct real numbers. For $f(x)=\mathbf{1}_{\{x_i\}}(x)$, 
\begin{equation}
   \exists g:\R\to\R_+\mbox{ measurable such that }\mu(dx)\mbox{-a.e.},\quad f(x)=\int_{y\in \mathcal{Y}}g(y)\,\pi_{X=x}(dy)\label{egfdon}
\end{equation}
 implies that $\sum_{j=1}^Jg(y_j)\,\pi_{X=x_i}(\{y_j\})=1$ and $\sum_{j=1}^Jg(y_j)\sum_{i'\ne i}\pi_{X=x_{i'}}(\{y_j\})=0$ so that there exists a function $\tau : \{1, \dots, I\} \to \{1, \dots, J\}$,
 such that $\pi_{X=x_i}(\{y_{\tau(i)}\})>0$ and $\sum_{i'\ne i}\pi_{X=x_{i'}}(\{y_{\tau(i)}\})=0$. Moreover $\tau$ is one-to-one so that $J\ge I$. Let us conversely suppose 
\begin{align}
   \exists\tau:\{1,\ldots,I\}&\rightarrow\{1,\ldots,J\}\mbox{ one-to-one such that }\notag\\&\forall i\in\{1,\ldots,I\},\;\pi_{X=x_i}(\{y_{\tau(i)}\})>0\mbox{ and }\sum_{i'\ne i}\pi_{X=x_{i'}}(\{y_{\tau(i)}\})=0.\label{cnssupfin}
\end{align}
Then $\sum_{i'\ne i}\pi_{X=x_i}(\{y_{\tau(i')}\})=0$ and for $f\in L^1_+(\mu)$, the function $g(y)=\sum_{i'=1}^I\mathbf{1}_{\{y=y_{\tau(i')}\}}\frac{f(x_{i'})}{\pi_{X=x_{i'}}(\{y_{\tau(i')}\})}$ is non-negative and such that for all $i\in\{1,\ldots,I\}$,
$$\int_{y \in \mathcal{Y}}g(y)\,\pi_{X=x_i}(dy)=\sum_{j=1}^Jg(y_j)\,\pi_{X=x_i}(\{y_j\})=\sum_{i'=1}^I\frac{f(x_{i'})\,\pi_{X=x_i}(\{y_{\tau(i')}\})}{\pi_{X=x_{i'}}(\{y_{\tau(i')}\})}=f(x_i).$$
Hence, when $\mu$ and $\nu$ are both finitely supported, \eqref{cnssupfin} is a necessary and sufficient condition for \eqref{egtoutf} to hold. In other words, \eqref{egtoutf} holds if and only if the matrix \(M = (M_{i,j})_{1 \leq i \leq I, 1 \leq j \leq J}\) defined by \(M_{i,j} = \pi_{X=x_i}(\{y_j)\}\) is of the following form (here for $I=4$ and $J=6$) 
\[
M = \begin{pmatrix}
0 & \textcolor{red}{\diamond} & \star & 0 & \star & 0 \\
\textcolor{red}{\diamond} & 0 & \star & 0 & \star & 0 \\
0 & 0 & \star & \textcolor{red}{\diamond} & \star & 0 \\
0 & 0 & \star & 0 & \star & \textcolor{red}{\diamond} \\
\end{pmatrix},
\]
where $\star$ are free non-negative coefficients (as long as the constraint that the sum of the coefficients in each row equals 1 is respected), and the $\textcolor{red}{\diamond}$ entries (those with indices $i,\tau(i)$) are strictly positive. 

Since for $i,i'\in\{1,\ldots,I\}$, $\pi_{X=x_{i'}}(\{y_{\tau(i)}\})=\frac{\pi(\{(x_{i'},y_{\tau(i)})\})}{p_{i'}}=\frac{q_{\tau(i)}}{p_{i'}}\,\pi_{Y=y_{\tau(i)}}(\{x_{i'}\})$, the condition\\ $\sum_{i'\ne i}\pi_{X=x_{i'}}(\{y_{\tau(i)}\})=0$ also writes $\pi_{Y=y_{\tau(i)}}=\delta_{x_i}$ and implies $\pi_{X=x_i}(\{y_{\tau(i)}\})>0$. In other words, \eqref{egtoutf} holds if and only if the matrix \(M^* = (M^*
_{j,i})_{1 \leq j \leq J, 1 \leq i \leq I}\) defined 
by \(M^*_{j,i} = \pi_{Y=y_j}(\{x_i\}) \) and obtained by dividing each row of the transpose of $M$ by the sum of its entries is of the following form
\[
M^* = \begin{pmatrix}
0 & \textcolor{red}{1} & 0 & 0\\
\textcolor{red}{1} & 0 & 0& 0 \\
\star & \star & \star & \star \\
0 & 0 & \textcolor{red}{1} & 0 \\
\star & \star & \star & \star \\
0 & 0 & 0 & \textcolor{red}{1} \\
\end{pmatrix}.
\]
It is therefore natural to introduce the Borel set $$D=\{y\in \mathcal{Y}:\pi_{Y=y}\mbox{ is a Dirac mass}\}.$$
We see that when $\mu$ and $\nu$ are finitely supported, then \eqref{cnssupfin} is equivalent to
\begin{align}
   \mu(dx)\text{-a.e.},\quad \pi_{X=x}(D)>0. \tag{$\mathcal{D}$}\label{cs}
\end{align}
This condition implies that $\nu(D)=\pi(\R\times D)=\int_{x\in\R}\pi_{X=x}(D)\,\mu(dx)>0.$ Not surprisingly, \eqref{cs} turns out to be sufficient for \eqref{egtoutf} to hold in general: intuitively, under \eqref{cs} one can choose the image of $\{y\in D:\pi_{Y=y}=\delta_{x}\}$ by $g$ to ensure that $\E[g(Y)|X=x]=f(x)$ and this does not influence $\E[g(Y)|X=x']$ for $x'\in\R\setminus\{x\}$.  This intuition is formalized in Proposition \ref{propcs}
below. The necessity of the sufficient condition \eqref{cs} is therefore a very natural question. It turns out that \eqref{cs} and \eqref{egtoutf} are equivalent when $\mu$ or $\nu$ are discrete. A more general framework ensuring equivalence is given in Corollary \ref{coreq}. But we are able to exhibit a probability measure $\pi$ on $\R^2$ such that \eqref{egtoutf} holds while $\nu(D)=0$, which shows that the sufficient condition \eqref{cs} is not always necessary. In this example, $\nu(dy)$-a.e., $\pi_{Y=y}$ is supported on two points and $\mu(dx)$-a.e., $\pi_{X=x}$ gives weight to $y$s such that these two points are arbitrarily close. The intuition that for \eqref{egtoutf} to hold, $\mu(dx)$-a.e., $\pi_{X=x}$ needs to give weight to points $y$ such that the support of $\pi_{Y=y}$ is ``not too large'' is confirmed by the fact that \eqref{egtoutf} does not hold when\begin{itemize}
\item  $\nu(dy)$-a.e., the component of $\pi_{Y=y}$ is absolutely continuous with respect to the Lebesgue measure is non-zero,
\item or for some $\varepsilon>0$,\, $\mu(\{x\in\R : \pi_{X=x}(\{y\in \mathcal{Y}:\pi_{Y=y}((x_\varepsilon,x_\varepsilon+\varepsilon])=1\mbox{ for some }x_\varepsilon\in\R\}) = 0\}) > 0$.
 \end{itemize} 

Our main results are stated and proved in Section 2. We may observe that \eqref{egtoutf} implies in particular the surjectivity of the operator
\[
T : L^1(\nu) \to L^1(\mu),
\]
defined by
\begin{equation}
Tg(x) = \int_{y \in \mathcal{Y}} g(y) \, \pi_{X=x}(dy) = \mathbb{E}[g(Y) | X = x].
\label{eq:def_T_operator}
\end{equation}
This motivates us, in Section~3, to relax the sign constraint on \( g \) and to study the surjectivity of \( T \) independently. It allows us to better understand which constraints on \( \pi \) arise from surjectivity alone, or from the sign constraint, for \eqref{egtoutf} to hold.
 
\section{Surjectivity with sign constraint}
\subsection{Preliminary results}
Before stating our main results, let us first state and prove some elementary preliminary results, including the general sufficiency of the condition  \eqref{cs}.
\begin{prop2}\label{propcs}
  The condition \eqref{cs} implies the representation property \eqref{egtoutf}.
 \end{prop2}
 \begin{rk}\label{remnud1}
   When $\nu(D)=1$, then $\int_{x\in\R}\, \pi_{X=x}(D^c)\, \mu(dx)=\nu(D^c)=0$, so $\mu(\{x\in\R:\pi_{X=x}(D)=1\})=1$. This implies \eqref{cs}. Moreover, for $\phi:D\to\R$ such that $\pi_{Y=y}=\delta_{\phi(y)}$ when $y\in D$, we then have $\pi(\{(\phi(y),y):y\in D\})=1$.
 \end{rk}
 \begin{proof}
   Let $\phi:D\to\R$ be such that for $y\in D$, $\pi_{Y=y}=\delta_{\phi(y)}$. We have
 \begin{align*}
   \int_{x\in\R}\pi_{X=x}(D\setminus\phi^{-1}(\{x\}))\,\mu(dx)
   =&\int_{(x,y)\in\R\times \mathcal{Y}}\mathbf{1}_D(y)\,\mathbf{1}_{\{\phi(y)\ne x\}}\,\pi(dx,dy)\\
   =&\int_{y\in \mathcal{Y}}\pi_{Y=y}(\R\setminus\{\phi(y)\})\,\mathbf{1}_D(y)\,\nu(dy)=0,
 \end{align*}
 so that \begin{equation}
\mu(dx)\text{-a.e.,} \quad \pi_{X = x}\big(D \setminus \phi^{-1}(\{x\})\big) = 0
\label{intcal}
\end{equation} and $\pi_{X=x}(\phi^{-1}(\{x\}))=\pi_{X=x}(D)$. % Under \eqref{cs}, $\mu(dx)\text{-a.e.},\;\pi_{X=x}(\phi^{-1}(\{x\}))=\pi_{X=x}(D)>0$. 
 For $f\in L^1_+(\mu)$, the mesurable function $$g(y)=\frac{\mathbf{1}_D(y)f(\phi(y))}{\pi_{X=\phi(y)}(\phi^{-1}(\{\phi(y)\}))},$$ is non-negative and such that $\mathbf{1}_{\{\pi_{X=x}(D)>0\}}\,\mu(dx)\text{-a.e.}$,
\begin{align*}
  \int_{y\in \mathcal{Y}}g(y)\,\pi_{X=x}(dy)&=\int_{y\in \mathcal{Y}} \mathbf{1}_{\phi^{-1}(\{x\})}(y)\frac{f(\phi(y))}{\pi_{X=\phi(y)}(\phi^{-1}(\{\phi(y)\}))}\,\pi_{X=x}(dy)\\&=\frac{f(x)}{\pi_{X=x}(\phi^{-1}(\{x\}))}\int_{y\in \mathcal{Y}}\mathbf{1}_{\phi^{-1}(\{x\})}(y)\,\pi_{X=x}(dy)=f(x),
\end{align*}
while $\mathbf{1}_{\{\pi_{X=x}(D)=0\}}\,\mu(dx)\text {-a.e.}$, $\int_{y\in \mathcal{Y}}g(y)\,\pi_{X=x}(dy)=0$.
Hence $$\mu(dx)\text{-a.e.,} \quad \int_{y\in \mathcal{Y}} g(y)\, \pi_{X=x}(dy)=\mathbf{1}_{\{\pi_{X=x}(D)>0\}}\, f(x),$$
and $\int_{y\in \mathcal{Y}}g(y)\,\nu(dy)=\int_{x\in\R}\mathbf{1}_{\{\pi_{X=x}(D)>0\}}\, f(x)\,\mu(dx)\le \int_{x\in\R} f(x)\,\mu(dx)$ so that $g\in L^1_+(\nu)$.
Under \eqref{cs}, $\mu(dx)\text{-a.e.,}$ $\mathbf{1}_{\{\pi_{X=x}(D)>0\}}=1$ and  \eqref{egtoutf} holds.\end{proof}
\begin{rk} Note that under \eqref{cs}, $g$ may not be unique. 
   If there exists $\tilde D\in\mathcal{B}(\R)$ such that $\tilde D\subset D$, $\nu(\tilde D)<\nu(D)$ and $\mu(dx)$\text{-a.e.}, $\pi_{X=x}(\tilde D)>0$, then $g(y)=\frac{\mathbf{1}_{\tilde D}(y)\,f(\phi(y))}{\pi_{X=\phi(y)}(\tilde D\cap \phi^{-1}(\{\phi(y)\}))}$  also is such that \(\mu(dx)\)\text{-a.e.},\;\(f(x)=\int_{y\in \mathcal{Y}}g(y)\pi_{X=x}(dy).\)
\end{rk}
\begin{prop2}\label{propcndisc} Under \eqref{egtoutf}, for all $x\in\R$ such that $\mu(\{x\})>0$, $\pi_{X=x}(D)>0$.
 \end{prop2}
 % Combining Propositions \ref{propcs} and \ref{propcndisc}, we immediately deduce the following Corollary.
 % \begin{cor2}
 %   When $\mu$ is discrete, then \eqref{cs}$\Leftrightarrow$\eqref{egtoutf}.
 % \end{cor2}
 \begin{cor2} \label{defC}
If the Borel set
$
C = \{x \in \mathbb{R} : \mu(\{x\}) + \pi_{X=x}(D) = 0\}
$
satisfies \(\mu({C})=0\), in particular when $\mu$ is discrete, then \eqref{cs} $\Longleftrightarrow$ \eqref{egtoutf}.
\end{cor2}
\begin{proof}[Proof of Corollary \ref{defC}]
We have $$\mu\left(\{x \in \mathbb{R} : \pi_{X=x}(D) = 0\}\right)=\mu(C)+\mu\left(\{x \in \mathbb{R} : \mu(\{x\})>0\mbox{ and }\pi_{X=x}(D) = 0\}\right).$$
If $\mu(C)=0$, then either $\mu\left(\{x \in \mathbb{R} : \mu(\{x\})>0\mbox{ and }\pi_{X=x}(D) = 0\}\right)>0$, in which case, according to Proposition \ref{propcndisc}, \eqref{egtoutf} does not hold, or $\mu\left(\{x \in \mathbb{R} : \pi_{X=x}(D) = 0\}\right)=0$, i.e., \eqref{cs} holds.
\end{proof}   \begin{proof}[Proof of Proposition \ref{propcndisc}]
 Let $z\in\R$ be such that $\mu(\{z\})>0$. Then \eqref{egfdon} for the choice $f(x)=\mathbf{1}_{\{z\}}(x)$ ensures that $B=g^{-1}((0,+\infty))$ is such that
 \[
 \pi_{X=z}(B)>0 \text{ and } \int_{x\in\R} \mathbf{1}_{\{x\ne z\}}\,\pi_{X=x}(B)\,\mu(dx)=0. 
 \]
 Then 
 \[
 \nu(B)=\pi(\R\times B)=\int_{x\in\R}\pi_{X=x}(B)\, \mu(dx)=\pi_{X=z}(B)\,\mu(\{z\})=\pi(\{z\}\times B)
 \]
 so that $\mathbf{1}_B(y)\,\nu(dy)=\mu(\{z\})\,\mathbf{1}_B(y)\, \pi_{X=z}(dy)$. 
 As a consequence, \begin{align*}
   0=\pi(\{\R\setminus\{z\}\}\times B)=\int_{y\in \mathcal{Y}}\pi_{Y=y}(\R\setminus\{z\})\,\mathbf{1}_{B}(y)\, \nu(dy)= \mu(\{z\})\int_{y\in \mathcal{Y}}\pi_{Y=y}(\R\setminus\{z\})\,\mathbf{1}_{B}(y)\, \pi_{X=z}(dy),\end{align*} and $\mathbf{1}_{B}(y)\, \pi_{X=z}(dy)$-a.e., $\pi_{Y=y}=\delta_{z}$, i.e., $y\in D$. This ensures that $\pi_{X=z}(D)\ge \pi_{X=z}(B)>0$.
\end{proof}
\begin{prop2}\label{propcneps}
If for some \(\varepsilon > 0\), \(\mu(\{x \in \mathbb{R} : \mu(\{x\}) + \pi_{X=x}(D_\varepsilon) = 0\}) > 0\) with $$D_\varepsilon=\{y\in \mathcal{Y}:\pi_{Y=y}((x_\varepsilon,x_\varepsilon+\varepsilon])=1\mbox{ for some }x_\varepsilon\in\R\},$$ then \eqref{egtoutf} does not hold.
\end{prop2}
Note that the same statement with $D_\varepsilon$ replaced by $D$ would, combined with Proposition \ref{propcs} and Corollary \ref{defC}, ensure the equivalence between \eqref{cs} and \eqref{egtoutf}. Of course, $\mu(\{x \in \mathbb{R}:\pi_{X=x}(D_\varepsilon)=0\})>0$ is stronger than $\mu(\{x \in \mathbb{R}:\pi_{X=x}(D)=0\})>0$. % \textcolor{blue}{Note that $\{x : \mu(\{x\}) + \pi_{X=x}(D_\varepsilon) = 0\} \subset C$.} 
\begin{cor2} \label{Depsilon}
If there exists $\varepsilon>0$ such that $\mathbf{1}_{\{\pi_{X=x}(D)=0\}}\,\mu(dx)$-a.e., $\pi_{X=x}(D_\varepsilon)=0$, and in particular when $\nu(D_\varepsilon\setminus D)=0$ for some $\varepsilon>0$, then \eqref{cs} $\Longleftrightarrow$ \eqref{egtoutf}.
\end{cor2}
\begin{proof}[Proof of Proposition \ref{propcneps}]
There exists some \(k \in \mathbb{Z}\) such that the Borel set
\[
A = (k\varepsilon, (k+1)\varepsilon] \cap \{x \in \mathbb{R} : \mu(\{x\}) + \pi_{X=x}(D_\varepsilon) = 0\} ,
\]
satisfies \(\mu(A) > 0\). Moreover,
\[
0 = \int_{x \in A} \pi_{X=x}(D_\varepsilon) \, \mu(dx) = \pi(A \times D_\varepsilon) = \int_{y \in \mathcal{Y}} \pi_{Y=y}(A)\,\mathbf{1}_{D_\varepsilon}(y) \, \nu(dy).
\]
Since, by the definition of \(D_\varepsilon\), \(\pi_{Y=y}((k\varepsilon, (k+1)\varepsilon]) < 1\) for \(y \notin D_\varepsilon\), we deduce that \(\nu(dy)\)-a.e., \(\pi_{Y=y}(A) < 1\).
Proposition \ref{propcn} ensures that \eqref{egtoutf} does not hold.\end{proof}

\subsection{Main results}
We recall that $C = \{x \in \mathbb{R} : \mu(\{x\}) + \pi_{X=x}(D) = 0\}$ and
$D_\varepsilon=\{y\in \mathcal{Y}:\pi_{Y=y}((x_\varepsilon,x_\varepsilon+\varepsilon])=1\mbox{ for some }x_\varepsilon\in\R\}$ for $\varepsilon>0$.\begin{thm2} \label{mainsuffisant}
Assume that there exist \(\varepsilon > 0\) and
\begin{itemize}
    \item a diffuse probability measure \(\eta\) on $\mathbb{R}$,
    \item a subset ${\mathcal X}$ of $\R$ with at most countable closure,
\end{itemize}
such that \(\mathbf{1}_{\{y \in D_{\varepsilon} \cap D^c:\text{ } \pi_{Y=y}(C) = 1\}}\, \nu(dy)\)-a.e.,
at least one of the following holds:
\begin{itemize}
    \item[(i)] the cumulative distribution function \(F_{Y=y}\) of $\pi_{Y=y}$ is increasing on some interval of positive length,
    \item[(ii)] there exists a probability measure $\gamma_y$ on $\mathbb{R},\text{ such that }
\gamma_y \ll \pi_{Y=y}\text{ and }
\gamma_y \ll \eta,$
    \item[(iii)] there exists \(x\in{\mathcal X}\) such that \(F_{Y=y}(x) \in (0, 1)\).
\end{itemize}
Then $\mu(C)>0$ implies that \eqref{egtoutf} does not hold. 
\end{thm2}
The proof of Theorem \ref{mainsuffisant} is given in Section \ref{sec:mainsuffisant}.
Suppose that $\nu(dy)\text{-a.e.}$, the component of $\pi_{Y=y}$ absolutely continuous with respect to the Lebesgue measure is non-zero.
Then $\mu$ is not discrete, so that $\mu(C)>0$. We can thus choose $\eta$ as the standard Gaussian distribution, to deduce the following corollary.
\begin{cor2}\label{Lebesguedensity}
   If $\nu(dy)\text{-a.e.}$, the component of $\pi_{Y=y}$ absolutely continuous with respect to the Lebesgue measure is non-zero, in particular when ${\mathcal Y}=\R^d$ and $\pi$ has a density with respect to the Lebesgue measure on $\R^{1+d}$, then \eqref{egtoutf} does not hold.
\end{cor2}
Using Proposition \ref{propcs} and Corollary \ref{defC}, we also deduce the following corollaries.
\begin{cor2}\label{coreq}
   Under the assumptions of Theorem \ref{mainsuffisant}, \eqref{cs} \(\Longleftrightarrow\) \eqref{egtoutf}.
 \end{cor2}
 
 \begin{cor2}\label{nudisc}
   If \( \nu \) is discrete, then \eqref{cs} \(\Longleftrightarrow\) \eqref{egtoutf}.
 \end{cor2}
 \begin{rk}\label{munudiscrete}
  Under \eqref{cs}, \(\mu\) is equivalent to \(\phi_{\#}(\mathbf{1}_D \, \nu)   \) for $\phi:D\to\R$ such that for $y\in D$, $\pi_{Y=y}=\delta_{\phi(y)}$. Indeed, applying \eqref{intcal} for the third and last equalities, we have for $A \in \mathcal{B}(\mathbb{R})$,\begin{align*}
\phi_{\#}(\mathbf{1}_D \, \nu)(A) &= \nu(\phi^{-1}(A)) = \int_{x \in \mathbb{R}} \pi_{X=x}(\phi^{-1}(A)) \, \mu(dx) = \int_{x \in \mathbb{R}} \pi_{X=x}(\phi^{-1}(A) \cap \phi^{-1}(\{x\})) \, \mu(dx)\\
&= \int_{x \in A} \pi_{X=x}(\phi^{-1}(\{x\})) \, \mu(dx)=\int_{x \in A} \pi_{X=x}(D) \, \mu(dx).
\end{align*}
With Corollary \ref{nudisc}, we deduce that if $\nu$ is discrete and (\ref{egtoutf}) holds, then $\mu$ is discrete.
 
 \end{rk}\begin{proof}[Proof of Corollary \ref{nudisc}]
   Either \( \mu \) is discrete and the equivalence follows from Corollary \ref{defC} or its diffuse component $\mathbf{1}_{\{\mu(\{x\})=0\}}\,\mu(dx)$ is non-zero. In the latter case, since
    $$
    \nu(dy)\text{-a.e.,}\quad \pi_{Y=y} \ll \mu,
    $$
    it comes from the definition of $C$ that $$\mathbf{1}_{\{y \in D^c:\ \pi_{Y=y}(C) = 1\}}\,\nu(dy)\text{-a.e.,}\quad \pi_{Y=y}(dx) \ll \mathbf{1}_{\{\mu(\{x\})=0\}}\,\mu(dx),$$
 and the equivalence follows from Corollary \ref{coreq} applied with $\eta(dx)=\frac{\mathbf{1}_{\{\mu(\{x\})=0\}}\mu(dx)}{\mu\left(\{z \in \mathbb{R}:\mu(\{z\})=0\}\right)}$ and $\gamma_y=\pi_{Y=y}$.
 \end{proof}
According to the next result, the general equivalence of \eqref{cs} and \eqref{egtoutf} fails.
\begin{thm2} \label{conter}
There exists a probability measure $\pi$ on $\R^2$ % with marginals $\mu$ and $\nu$ absolutely continuous with respect to the Lebesgue measure (in particular 
 such that $
\nu(D) = 0
$ (in particular (\ref{cs}) does not hold), $\mu$ and $\nu$ are absolutely continuous with respect to the Lebesgue measure, and
\eqref{egtoutf} holds.
\end{thm2}
The proof of Theorem \ref{conter} is given in Section \ref{sec:conter}.
\subsection{Proof of Theorem \ref{mainsuffisant}}\label{sec:mainsuffisant}
In the proof, we will contradict the following necessary condition for \eqref{egtoutf} to hold.\begin{prop2}
\label{propcn}   % Let $A\in{\mathcal B}(\R)$ be such that $\mu(A)>0$. When \eqref{egfdon} holds for $f=1_A$, then  $\nu\left(\{y\in\R:\pi_{Y=y}(A)=1\}\right)>0$.
If \eqref{egtoutf} holds, then for any Borel subset $A$ of $\{x\in\R:\pi_{X=x}(D)=0\}$ such that $\mu(A)>0$, we have $\nu\left(\{y\in D^c:\pi_{Y=y}(A)=1\}\right)>0$.
 \end{prop2}
\begin{proof}
   For $f=\mathbf{1}_A$, \eqref{egfdon} implies on the one hand that $\mathbf{1}_A(x)\,\mu(dx)$-a.e., 
    $$\pi_{X=x}(g^{-1}((0,+\infty))\cap D^c)=\pi_{X=x}(g^{-1}((0,+\infty)))>0$$
   so that 
    $$\nu(g^{-1}((0,+\infty))\cap D^c)\ge \int_{x\in A}\pi_{X=x}(g^{-1}((0,+\infty))\cap D^c)\,\mu(dx)>0.$$
On the other hand, $\mathbf{1}_{A^c}(x)\,\mu(dx)$-a.e., 
   $$\pi_{X=x}(g^{-1}((0,+\infty))\cap D^c)\le \pi_{X=x}(g^{-1}((0,+\infty)))=0$$ so that 
    $$\int_{y\in g^{-1}((0,+\infty))\cap D^c}\,\pi_{Y=y}(A^c)\,\nu(dy)=\int_{x\in A^c}\pi_{X=x}(g^{-1}((0,+\infty))\cap D^c)\,\mu(dx)=0,$$
    and $\mathbf{1}_{g^{-1}((0,+\infty))\cap D^c}(y)\,\nu(dy)$-a.e.,\quad$\pi_{Y=y}(A)=1-\pi_{Y=y}(A^c)=1$.
  \end{proof}
To contradict the necessary condition, we combine the following propositions which respectively deal with the three alternative properties satisfied \(\mathbf{1}_{\{y \in D_{\varepsilon} \cap D^c:\text{ } \pi_{Y=y}(C) = 1\}}\, \nu(dy)\)-a.e.\ in the assumptions of  Theorem \ref{mainsuffisant}.

  \begin{prop2} \label{croisssur}  Suppose that $\tilde A$ is a Borel subset of $C$ such that $\mu(\tilde A)>0$ and $B$ is a Borel subset of $\{y\in D^c:\pi_{Y=y}(\tilde A)=1\}$ such that $\mathbf{1}_{B}(y)\, \nu(dy)$-a.e., there exists an interval of positive length on which the cumulative distribution function of \(\pi_{Y=y}\), denoted by $F_{Y=y}$, is increasing. Then there exists a Borel subset $A$ of $\tilde A$ such that $\mu(A)>0$ and $\mathbf{1}_{B}(y)\, \nu(dy)$\text{-a.e.,} $\pi_{Y=y}(A)<1$.
\end{prop2}
\begin{proof}
  If $\nu(B)=0$, then we can choose $A=\tilde A$. Let us now suppose that $\nu(B)>0$. Let \((q_n)_{n \in \mathbb{N}}\) be a sequence that runs through \(\mathbb{Q}\).  
Since $\tilde A\subset C$, the probability measure 
\[
\mu_{\tilde A}(dx) = \frac{\mathbf{1}_{\tilde A}(x)}{\mu(\tilde A)} \mu(dx)
\]
has no atom and its cumulative distribution function is continuous. Hence, there exists a sequence \((\varepsilon_n)_{n \in \mathbb{N}}\) such that
\[
\forall n \in \mathbb{N},\text{ } \mu_{\tilde A}((q_n - \varepsilon_n, q_n + \varepsilon_n)) 
\leq  \frac{1}{2^{n+2}}.
\]
We define
\[
A = {\tilde A} \setminus \bigcup_{n \in \mathbb{N}^*} (q_n - \varepsilon_n, q_n + \varepsilon_n).
\]
We have
\begin{align*}
\mu(A) &\geq \mu({\tilde A}) \, \mu_{\tilde A}(A) \\
&= \mu({\tilde A}) 
\left(1 - \mu_{\tilde A}(\bigcup_{n \in \mathbb{N}} (q_n - \varepsilon_n, q_n + \varepsilon_n))\right) \\
&\geq \mu({\tilde A}) 
\left(1 - \sum_{n \in \mathbb{N}} \frac{1}{2^{n+2}}\right) \\
&= \frac{\mu({\tilde A})}{2} > 0.
\end{align*}
By assumption $\mathbf{1}_B(y)\, \nu(dy)$-a.e., there exist \(a_y < b_y\) such that \(F_{Y=y}\) is increasing on \([a_y, b_y]\). By density of \(\mathbb{Q}\) in \(\mathbb{R}\), there exists \(n_y \in \mathbb{N}\) such that \(a_y \leq q_{n_y} \leq b_y\). \\The intersection of \((q_{n_y} - \varepsilon_{n_y}, q_{n_y} + \varepsilon_{n_y})\) with \([a_y, b_y]\) is an interval of positive length, so that  \[
\pi_{Y=y}\left((q_{n_y} - \varepsilon_{n_y}, q_{n_y} + \varepsilon_{n_y})\right) > 0 \quad \mbox{and} \quad  \pi_{Y=y}(A)<1.\]
\end{proof}
\begin{prop2} \label{singlar} Let \( \eta \) be a diffuse probability measure and $(\gamma_y)_{y \in \mathcal{Y}}$ be a family of probability measures. Suppose that $\tilde A$ is a Borel subset of $C$ such that $\mu(\tilde A)>0$ and $B$ is a Borel subset of $\{y\in D^c:\pi_{Y=y}(\tilde A)=1\}$ such that $\mathbf{1}_{B}(y)\, \nu(dy)\text{-a.e., } \gamma_y \ll \pi_{Y=y} \text{ and }
\gamma_y \ll \eta.$ Then there exists a Borel subset $A$ of $\tilde A$ such that $\mu(A)>0$ and $\mathbf{1}_{B}(y)\, \nu(dy)$-a.e., $\pi_{Y=y}(A)<1$.
\end{prop2}
The proof of Proposition \ref{singlar} is given in Section \ref{sec:singlar}.
\begin{prop2}\label{hhhhhhh}
Let ${\mathcal X}$ be a subset of $\mathbb{R}$ with at most countable closure.
 Suppose that $\tilde A$ is a Borel subset of $C$ such that $\mu(\tilde A)>0$ and $B$ is a Borel subset of $\{y\in D^c:\pi_{Y=y}(\tilde A)=1\}$ such that $\mathbf{1}_{B}(y)\,\nu(dy) \text{-a.e., there exists } x\in{\mathcal X}\text{ such that } F_{Y=y}(x) \in (0,1)$. Then there exists a Borel subset $A$ of $\tilde A$ such that $\mu(A)>0$ and $\mathbf{1}_{B}(y)\, \nu(dy)$-a.e., $\pi_{Y=y}(A)<1$.

\end{prop2} 
The proof of Proposition \ref{hhhhhhh} is given in Section \ref{sec:hhhhhhh}.
\begin{proof}[Proof of Theorem \ref{mainsuffisant}]
 
Let us suppose that $\mu(C)>0$ and exhibit a Borel subset \(A\) of \(C\) such that
\[
\mu(A) > 0 \quad \text{and} \quad \nu\left(\{y \in D^c : \pi_{Y=y}(A) = 1\}\right) = 0,
\]
so that, by Proposition \ref{propcn}, \eqref{egtoutf} does not hold.\\
\textbf{Step 0:} Like in the proof of Proposition \ref{propcneps}, there exists some \(k \in \mathbb{Z}\) such that the Borel set
\(A_0 = (k\varepsilon, (k+1)\varepsilon] \cap C\) satisfies $\mu(A_0)>0$. By definition of $D_{\varepsilon}$, \(\mathbf{1}_{D^c \setminus D_{\varepsilon}}(y)\, \nu(dy)\)-a.e., \(\pi_{Y=y}(A_0) < 1\). Moreover, since $A_0\subset C$, $\mathbf{1}_{\{y \in D_\varepsilon\cap D^c:\text{ } \pi_{Y=y}(C) < 1\}} \, \nu(dy)\text{-a.e., } \pi_{Y=y}(A_0) \leq \pi_{Y=y}(C) < 1.$
\\
\textbf{Step 1:} Let \(B_1 = \{y \in D_{\varepsilon} \cap D^c : \pi_{Y=y}(C) = 1 \text{ and } F_{Y=y} \text{ is increasing on an interval with positive length} \}\).  
By Proposition \ref{croisssur} applied with $(\tilde A,B)=(A_0,B_1\cap\{y \in \mathcal{Y}:\pi_{Y=y}(A_0)=1\})$, we construct a Borel subset \(A_1\) of \(A_0\) such that \(\mu(A_1) > 0\) and \(\mathbf{1}_{B_1}(y)\, \nu(dy)\)-a.e., \(\pi_{Y=y}(A_1) < 1\).
\\
\textbf{Step 2:} Let \(B_2 = \{y \in D_{\varepsilon} \cap D^c : \pi_{Y=y}(C) = 1 \text{ and there exists a probability measure } \gamma_y\ll\eta, \text{ such that }
\gamma_y \ll \pi_{Y=y}\}\).  
By Proposition \ref{singlar} applied with $(\tilde A,B)=(A_1,B_2\cap\{y \in \mathcal{Y}:\pi_{Y=y}(A_1)=1\})$, we construct a Borel subset \(A_2\) of \(A_1\)  such that \(\mu(A_2) > 0\) and \(\mathbf{1}_{B_2}(y)\, \nu(dy)\)-a.e., \(\pi_{Y=y}(A_2) < 1\).
\\
\textbf{Step 3:} Let \(B_3 = \{y \in D_{\varepsilon} \cap D^c : \pi_{Y=y}(C) = 1 \text{ and } \exists x\in {\mathcal X}\;\text{such that }F_{Y=y}(x) \in (0, 1) \}\).  
By Proposition \ref{hhhhhhh} applied with $(\tilde A,B)=(A_2,B_3\cap\{y \in \mathcal{Y}:\pi_{Y=y}(A_2)=1\})$, we construct a Borel subset \(A\) of \(A_2\) such that \(\mu(A) > 0\) and \(\mathbf{1}_{B_3}(y)\, \nu(dy)\)-a.e., \(\pi_{Y=y}(A) < 1\). 

Since $A\subset A_2\subset A_1\subset A_0$, we then have, under the assumption of the theorem, $\mathbf{1}_{ D^c}(y) \, \nu(dy)\text{-a.e., } $
\begin{itemize}
\item  $y\in D^c \setminus D_{\varepsilon}$  and $\pi_{Y=y}(A)\le \pi_{Y=y}(A_0)<1$,
\item or $y\in D_\varepsilon\cap D^c$ satisfies $\pi_{Y=y}(C)<1$ and $\pi_{Y=y}(A)\le \pi_{Y=y}(A_0)<1$,
\item or $y\in B_1$ and $\pi_{Y=y}(A)\le \pi_{Y=y}(A_1)<1$,
\item or $y\in B_2$ and $\pi_{Y=y}(A)\le \pi_{Y=y}(A_2)<1$,
\item or $y\in B_3$ and $\pi_{Y=y}(A)<1$.\end{itemize}
\end{proof}
\label{dec}
The following proposition holds in a framework of limited operational value, in which \eqref{egtoutf} is equivalent to \eqref{propcs}. This result is used in the proof of Proposition \ref{hhhhhhh}.
\begin{prop2} \label{dec}
When there exists a countable family \((A_n)_{n \in \mathbb{N}}\) of disjoint Borel subsets of 
\(\{x \in \mathbb{R} : \pi_{X=x}(D) = 0\}\) such that 
\[
\sum_{n \in \mathbb{N}} \mu(A_n) = \mu(\{x \in \mathbb{R} : \pi_{X=x}(D) = 0\}),
\]
and for all \(n \in \mathbb{N}\), \(\mathbf{1}_{\{y \in D^c:\text{ } \pi_{Y=y}(C) = 1\}}\,\nu(dy)\)-a.e., \(\pi_{Y=y}(A_n) <1\), then $\mu(C)>0$ implies that \eqref{egtoutf} does not hold. Moreover, \eqref{cs} $\Longleftrightarrow$ \eqref{egtoutf}.
\end{prop2}
\begin{proof}
By Proposition \ref{propcs} and Corollary \ref{defC}, the second assertion follows from the first. 
Let us suppose that \(\mu(C)>0\).  By hypothesis, there exists $n \in \mathbb{N}$, such that $\mu(A_n\cap C)>0$ and
\[
\mathbf{1}_{\{y \in D^c:\text{ } \pi_{Y=y}(C) = 1\}} \, \nu(dy) \text{-a.e.,}\quad \pi_{Y=y}(A_n \cap C) < 1.
\]
As we also have
\[
\mathbf{1}_{\{y \in D^c:\text{ }  \pi_{Y=y}(C) < 1\}} \, \nu(dy) \text{-a.e.,}\quad \pi_{Y=y}(A_n \cap C) \leq \pi_{Y=y}(C) < 1,
\]
it follows that
\[
\mathbf{1}_{D^c}(y) \, \nu(dy) \text{-a.e.,}\quad \pi_{Y=y}(A_n \cap C) < 1.
\]
By Proposition \ref{propcn}, \eqref{egtoutf} does not hold.
\end{proof}\subsection{Proof of Proposition \ref{singlar}}\label{sec:singlar}The proof of Proposition \ref{singlar} relies on the following lemma, the proof of which is postponed.
\begin{lemma} \label{final}
Let $\eta$ be a diffuse probability measure. Let $(f_y)_{y\in \mathcal{Y}}$ be a family of non-negative measurable functions from $\mathbb{R}$ to $\mathbb{R}$ such that the mapping
$
(x,y) \mapsto f_y(x)
$
is measurable and
\[
\nu(dy)\text{-a.e.,}\quad\int_{x \in \mathbb{R}} f_y(x)\, \eta(dx)>0.
\]
Then for each $\widehat{A} \in \mathcal{B}(\mathbb{R})$ such that $\eta(\widehat{A})>0$, there exists a Borel set $A \subset \widehat{A}$ such that $\eta(A)>0$ and $\nu(dy)$-a.e.,
\[
\int_{x \in A} f_y(x)\, \eta(dx) < \int_{x \in \mathbb{R}} f_y(x)\, \eta(dx).
\]
\end{lemma}
 \begin{proof}[Proof of Proposition \ref{singlar}.]
If $\nu(B)=0$, then we can choose $A=\tilde A$. Let us now suppose that $\nu(B)>0$. Up to replacing $\eta$ by 
$
\frac{\eta + \mu(\cdot \cap C)}{1 + \mu(C)},
$
we can suppose that $\mu(\cdot \cap C) \ll \eta$. We define the Borel set
\[
\widehat{A}
 = \{ x \in \tilde A : f_\mu(x) > 0 \},
\]
where $f_\mu$ is the density of $\mu(\cdot \cap C)$ with respect to $\eta$. Since \( \mu(\tilde A ) >0 \), it follows that \( \eta(\widehat{A}) > 0 \). We also denote by  $f_y$ the density of $\gamma_y$ with respect to $\eta$.
 By Lemma \ref{final}, there exists a Borel set \( A \subset \widehat{A} \) such that \( \eta(A) > 0 \) so that \( \mu(A) > 0 \) by definition of $\widehat{A}$ and
$$
\mathbf{1}_{B}(y) \, \nu(dy)\text{-a.e.,} \quad
\gamma_{y}(A) =\int_{x \in A} f_y(x) \, \eta(dx) < \int_{x \in \mathbb{R}} f_y(x) \, \eta(dx) =1.
$$
Since $\mathbf{1}_{B}(y) \, \nu(dy)\text{-a.e.,}$ $\gamma_y \ll \pi_{Y=y}$, it comes that
$$
\mathbf{1}_{B}(y) \, \nu(dy)\text{-a.e.,} \quad
\pi_{Y=y}(A) <1.
$$
\end{proof}
The proof of Lemma \ref{final} relies on the following lemma. It can be shown that for any diffuse probability measure, one can construct a mixing sequence of Borel sets with arbitrary constant mass. See \cite{Renyi1958} for a definition of mixing sequences and a study of its properties. In the rest of the article, $\lambda$ denotes the Lebesgue measure on $\mathbb{R}$.
\begin{lemma}\label{independetsequence_0}
Let \( \eta \) be a diffuse probability measure on \( \mathbb{R}\) and $a \in (0,1)$. Then there exists a sequence of Borel sets $(A_m)_{m \in \mathbb{N}}$ such that
\begin{enumerate}
    \item[(i)] for each \( m \in \mathbb{N} \), \( \eta(A_m) = a \),
    \item[(ii)] for all \( m, m' \in \mathbb{N} \) with \( m \neq m' \), we have
    \[
    \eta(A_m \cap A_{m'}) = a^2,
    \]
    \item[(iii)] for every \( \widehat{A} \in \mathcal{B}(\mathbb{R}) \),
    \[
    \lim_{m \to \infty} \eta(\widehat{A} \cap A_m) = a \, \eta(\widehat{A}).
    \]
\end{enumerate}
\end{lemma}
\begin{proof}
Let \( a \in (0,1) \). We first construct a family \( (\tilde{A}_m)_{m \in \mathbb{N}} \) of sets in \( \mathcal{B}([0,1]) \) such that
\begin{itemize}
    \item for all \( m \in \mathbb{N} \), we have \( \lambda(\tilde{A}_m) = a \),
    \item for all \( m, m' \in \mathbb{N} \) with \( m \neq m' \), we have
    \[
    \lambda(\tilde{A}_m \cap \tilde{A}_{m'}) = a^2.
    \]
\end{itemize}
On the probability space $([0,1],{\mathcal B}([0,1]),\mathbf{1}_{[0,1]}\lambda)$, we define the random variable \(U : x\mapsto x\). For $n\geq 1$, let
$
D_n = \lfloor 10^n U \rfloor - 10 \lfloor 10^{n-1} U \rfloor \in \{0, \dots, 9\}
$
be the \(n\)-th digit of \(U\). 
Let \(\varphi : \mathbb{N} \times \mathbb{N}^* \to \mathbb{N}^*\) be the bijection \(\varphi(m, p) = 2^m (2p - 1)\). Then, for \(m \in \mathbb{N}\), we set
\[
U_m = \sum_{p=1}^{+\infty} \frac{D_{\varphi(m, p)}}{10^p}.
\]
The random variables \((U_m)_{m \in \mathbb{N}}\) are independent and uniformly distributed on \([0,1]\).  For \( m \in \mathbb{N} \), we consider the Borel sets $\tilde{A}_m=\{x \in [0,1] : U_m(x) \leq a\}$. Now let \( \eta \) be a diffuse probability measure on \( (\mathbb{R}, \mathcal{B}(\mathbb{R})) \) with cumulative distribution function \( F_\eta \). We have the well-known identity
\[
(F_\eta)_{\#} \eta = \mathbf{1}_{[0,1]}\, \lambda,
\]
i.e.,
\[
\forall A \in \mathcal{B}([0,1]),\,\
\lambda(A) = \eta(F_\eta^{-1}(A)).
\]
The sequence of Borel sets \( (A_m)_{m \in \mathbb{N}} \) defined for every \( m \in \mathbb{N} \), by
$
A_m = F_\eta^{-1}(\tilde{A}_m)
$
satisfies (i) and (ii) since $A_m\cap A_{m'}=F_\eta^{-1}(\tilde{A}_m\cap \tilde A_{m'})$.
On the probability space $(\mathbb{R}, \mathcal{B}(\mathbb{R}),\eta)$, the random variables \( (\mathbf{1}_{A_m})_{m\in\N} \) are i.i.d.\ and follow the Bernoulli distribution with parameter \( a \).
For $\widehat{A} \in \mathcal{B}(\mathbb{R})$, by Bessel's inequality in the Hilbert space $L^2(\mathbb{R}, \mathcal{B}(\mathbb{R}), \eta)$ with the orthonormal family $\left(\frac{\mathbf{1}_{A_m}-a}{\sqrt{a(1-a)}}\right)_{m \in \mathbb{N}}$, we have
\[
\sum_{m \in \mathbb{N}}\mathbb{E}\left[(\mathbf{1}_{\widehat A}-\eta(\widehat A))\frac{\mathbf{1}_{A_m}-a}{\sqrt{a(1-a)}} \right]^2
\leq \mathbb{E}[(\mathbf{1}_{\widehat A}-\eta(\widehat A))^2]= \eta (1-\eta) \le\frac 1 4.
\]
Thus $\lim_{m \to \infty}
\mathbb{E}[(\mathbf{1}_{\widehat A}-\eta(\widehat A))(\mathbf{1}_{A_m}-a) ]
=
\lim_{m \to \infty} \eta(\widehat{A} \cap {A}_m) - a\,\eta(\widehat{A})
=0.
$\end{proof}\begin{proof}[Proof of Lemma \ref{final}.]
Let \( q \in (0,1) \) be such that
$
q \leq \frac{\eta(\widehat{A})}{2 + \eta(\widehat{A})}
$. For $n \in \mathbb{N}^*$, by Lemma \ref{independetsequence_0} applied with $a$ replaced by 
$
a_n = 1-q^n$, there exists a sequence of Borel sets
$
(A_{m}^n)_{m \in \mathbb{N}}
$
such that 
\begin{align*}
&\forall m \in \mathbb{N},\quad  \eta(A_{m}^n) = a_n, \\
&\forall \tilde{A} \in \mathcal{B}(\mathbb{R}),\quad  \lim_{m \to \infty} \eta(\tilde{A} \cap A_{m}^n) = a_n \, \eta(\tilde{A}).
\end{align*}
We have $\nu(dy)$-a.e., 
\[
\lim_{m \to \infty} \eta\Bigl(\{x \in \mathbb{R} : f_y(x) > 0\} \cap A_{m}^n\Bigr)
= a_n\, \eta\Bigl(\{x \in \mathbb{R} : f_y(x) > 0\}\Bigr) 
< \eta\Bigl(\{x \in \mathbb{R} : f_y(x) > 0\}\Bigr),
\]
so that
\[
\lim_{m \to \infty} \nu\left(\left\{ y \in \mathcal{Y} : \int_{x \in A_{m}^n} f_{y}(x) \, \eta(dx) = \int_{x \in \mathbb{R}} f_y(x) \, \eta(dx) \right\} \right) = 0.
\]
We can thus consider \( m_n \in \mathbb{N} \) such that the Borel set
\[
B_n=\left\{ y \in \mathcal{Y} : \int_{x \in A_{m_n}^n} f_y(x) \, \eta(dx) = \int_{x \in \mathbb{R}} f_y(x)\, \eta(dx) \right\},
\]
satisfies
$
\nu \left(B_n\right) < \frac{1}{n}.
$
Let $\widehat{A} \in \mathcal{B}(\mathbb{R})$ be such that $\eta(\widehat{A})>0$. We define the Borel set
\[
A = \widehat{A} \cap \bigcap_{n \in \mathbb{N}^*} A_{m_n}^n.
\]
Based on the definition of \( A_{m_n}^n \), we have
\begin{align*}
\eta(A) &= 1 - \eta(A^c) = 1 - \eta\Bigl(\widehat{A}^c \cup \bigcup_{n \in \mathbb{N}^*} (A_{m_n}^n)^c\Bigr) \\
           &\geq 1 - \eta(\widehat{A}^c) - \sum_{n \in \mathbb{N}^*} \eta\Bigl((A_{m_n}^n)^c\Bigr)\\
            &= 1 -(1-\eta(\widehat{A})) - \sum_{n \in \mathbb{N}^*} \left(1 - \eta\left(A_{m_n}^n\right)\right) \\
           &=\eta(\widehat{A}) - \sum_{n \in \mathbb{N}^*} q^n \\
           &\geq \eta(\widehat{A})-\frac{\eta(\widehat{A})}{2} > 0.
\end{align*}
We define the Borel set
$
B= \bigcap_{n \in \mathbb{N}^*} B_n.
$
Let $y \in B^c$. There exists $n \in \mathbb{N}^*$ such that $y \notin B_{n}$, so that
\[
\int_{x \in A} f_y(x)\, \eta(dx) 
\leq \int_{ x \in {A_{m_{n_{}}} ^{n_{}}  }} f_y(x)\, \eta(dx) < \int_{x \in \mathbb{R}} f_y(x)\, \eta(dx).
\]
Therefore,
\[
B^c \subset \left\{ y \in \mathcal{Y} : \int_{x \in A} f_y(x)\, \eta(dx) < \int_{x \in \mathbb{R}} f_y(x)\, \eta(dx) \right\}.
\]
Moreover, for all $n \in \mathbb{N}^*,\,\nu(B) \leq \nu(B_n) \leq \frac{1}{n}.
$
It follows that \( \nu(B^c) = 1 \).
\end{proof}
\subsection{Proof of Proposition \ref{hhhhhhh}}\label{sec:hhhhhhh}
\begin{def2}\label{13}
\begin{itemize}
\item[(i)] \( x \in \mathbb{R} \) is called an accumulation point of \( {\mathcal X} \) if for every \(\varepsilon > 0\), the interval \((x-\varepsilon, x+\varepsilon)\)  contains infinitely many points of \( {\mathcal X}\).

\item[(ii)] \( x \in \mathbb{R} \) is called a right accumulation point of \( {\mathcal X} \) if for every \(\varepsilon > 0\), the interval \([x, x+\varepsilon)\) contains infinitely many points of \( {\mathcal X} \).

\item[(iii)] \( x \in \mathbb{R} \) is called a left accumulation point of \( {\mathcal X} \) if for every \(\varepsilon > 0\), the interval \((x-\varepsilon, x]\) contains infinitely many points of \( {\mathcal X} \).

% \item[(iv)] \(x \in \mathbb{R}\) is called a purely right accumulation point (resp. purely left accumulation point) of \({\mathcal X}\) if \(x\) is a right accumulation point of \({\mathcal X}\) but not a left accumulation point of \({\mathcal X}\) (resp. \(x\) is a left accumulation point of \({\mathcal X}\) but not a right accumulation point of \({\mathcal X}\)).
\end{itemize}
\end{def2}

In the following, $\bar{\mathcal X}$ denotes the closure of ${\mathcal X}$ and $\tilde{\mathcal X}=\{x\in\bar{\mathcal X}:x\mbox{ is not a right accumulation point of } {\mathcal X}\}$.
For \(x < \sup{\mathcal X}\) such that \(x\) is not a right accumulation point of \({\mathcal X}\), let \( L_x = \inf \left\{ y \in {\mathcal X} : x<y \right\} \) where the infimum is taken over a non-empty set.%  and $A_x=(x,L_x)$.
% \item Let $A_{-\infty}=(-\infty,\inf{\mathcal X})$ and $A_{+\infty}=(\sup{\mathcal X},+\infty)$
The proof of Proposition \ref{hhhhhhh} relies on the next two lemmas, the first of which is obvious.

\begin{lemma}\label{lemdisj}
   We have $\bar{\mathcal X}\cap(-\infty,\inf{\mathcal X})=\bar{\mathcal X}\cap(\sup{\mathcal X},+\infty)=\emptyset$. Moreover, for \(x < \sup{\mathcal X}\) such that \(x\) is not a right accumulation point of \({\mathcal X}\), we have $L_x> x$, $L_x\in\bar{\mathcal X}$ and $(x,L_x)\cap\bar{\mathcal X}=\emptyset$.
\end{lemma}
\begin{lemma} \label{rrrrrrr}
Under the assumptions of Proposition \ref{hhhhhhh}, the following assertions hold:
\begin{itemize}
\item[(i)] 
\[
\mathbf{1}_B(y) \, \nu(dy)\text{-a.e.,} \quad \pi_{Y=y}((\sup{\mathcal X}, +\infty)^c)  > 0.
\]
\item[(ii)]
\[
\mathbf{1}_B(y) \, \nu(dy)\text{-a.e.,}\quad  \pi_{Y=y}(( -\infty, \inf{\mathcal X})^c)  > 0.
\]
\item[(iii)] For \( x < \sup{\mathcal X}\) \text{such that} \(x\) \text{is not a right accumulation point of }\({\mathcal X}\),
\[
\mathbf{1}_B(y)\, \nu(dy)\text{-a.e.,} \quad \pi_{Y=y}((x, L_x)^c)>0.
\]
\item[(iv)] For all \( x \in \mathbb{R} \),
\[
\mathbf{1}_B(y)\, \nu(dy)\text{-a.e.,}\quad  \pi_{Y=y}\left(\{ x \}^c\right) > 0.
\]
\end{itemize}
\end{lemma}\begin{proof}
\begin{itemize}
\item[(i)-(ii)] By hypothesis, \( \mathbf{1}_B(y)\, \nu(dy) \)-a.e., there exists \( x_y \in {\mathcal X} \) such that \( F_{Y=y}(x_y) \in (0,1),\) so that, since $(-\infty,x_y]\subset (\sup{\mathcal X},+\infty)^c$ and\\ $(x_y,+\infty)\subset (-\infty,\inf{\mathcal X})^c$, $\pi_{Y=y}((\sup{\mathcal X},+\infty)^c)>0$ and $\pi_{Y=y}((-\infty,\inf{\mathcal X})^c)>0$.\item[(iii)]Let $x < \sup{\mathcal X}$ be such that \(x\) is not a right accumulation point of \({\mathcal X}\). By definition of $L_x$, $(x,L_x)\cap{\mathcal X}=\emptyset$.
By hypothesis, \( \mathbf{1}_B(y)\, \nu(dy)\)-a.e., there exists \( x_y \in {\mathcal X} \) such that \( F_{Y=y}(x_{y}) \in (0,1), \) and either $x_y\le x$ so that $(-\infty, x_y]\subset (x,L_x)^c$ and $\pi_{Y=y}((x,L_x)^c)\ge F_{Y=y}(x_y)>0$ or $x_y\ge L_x$ so that $(x_y,+\infty)\subset (x,L_x)^c$ and $\pi_{Y=y}((x,L_x)^c)\ge 1-F_{Y=y}(x_y)>0$.\item[(iv)] Let \( x \in \mathbb{R} \).
By hypothesis \( \mathbf{1}_B(y)\, \nu(dy) \)-a.e., there exists \( x_y\in {\mathcal X} \) such that \( F_{Y=y}(x_{y}) \in (0,1)\) and either $x\ge x_y$ so that $(-\infty,x_y]\subset \{x\}^c$ and $\pi_{Y=y}\left(\{ x \}^c\right)\ge F_{Y=y}(x_y)>0$ or $x<x_y$ so that $(x_y,+\infty)\subset \{x\}^c$ and $\pi_{Y=y}\left(\{ x \}^c\right)\ge 1-F_{Y=y}(x_y)>0$. \end{itemize}
\end{proof}
\begin{proof}[Proof of Proposition \ref{hhhhhhh}]
  If $\nu(B)=0$, then we can choose $A=\tilde A$. Let us now suppose that $\nu(B)>0$. By Lemma \ref{lemdisj}, the intervals 
  \begin{equation}
   (-\infty,\inf{\mathcal X}),((x,L_x))_{x\in \tilde{\mathcal X}},\{x\}_{x\in\bar{\mathcal X}},(\sup{\mathcal X},+\infty)\label{parti}
  \end{equation}
  are disjoint. To prove that they form an at most countable partition of $\R$, it is enough to check that if $x\in\R\setminus \bar{\mathcal X}$ satisfies $\inf{\mathcal X}<x<\sup{\mathcal X}$, then there exists $\tilde x\in \tilde{\mathcal X}$ such that $x\in (\tilde x,L_{\tilde x})$. We set $\tilde x=\sup\{y\in{\mathcal X}:y<x\}$ where the supremum is taken over a non-empty set since $\inf{\mathcal X}<x$. Either $\tilde x\in {\mathcal X}$ or $\tilde x$ is a left accumulation point of ${\mathcal X}$. In particular, $\tilde x\in\bar{\mathcal X}$. Since $x\notin\bar{\mathcal X}$, we deduce that $\tilde x<x$. Moreover, by definition of $\tilde x$, $(\tilde x,x)\cap{\mathcal X}=\emptyset$ so that $(\tilde x,x)\cap\bar{\mathcal X}=\emptyset$ since $(\tilde x,x)$ is open and $(\tilde x,x]\cap\bar{\mathcal X}=\emptyset$ since $x\notin\bar{\mathcal X}$. This implies that $\tilde x$ is not a right accumulation point of ${\mathcal X}$ and $\tilde x\in\tilde{\mathcal X}$. Since, by Lemma \ref{lemdisj}, $L_{\tilde x}>\tilde x$ and $L_{\tilde x}\in\bar{\mathcal X}$, we also deduce from the equality $(\tilde x,x]\cap\bar{\mathcal X}=\emptyset$ that $L_{\tilde x}>x$.

  We conclude that the collection of intervals in \eqref{parti} is an at most countable partition of the real line.
  Since $\mu(\tilde{A}) > 0$, there exists an interval $I$ in this partition such that $A=\tilde{A}\cap I$ satisfies $\mu({A}) > 0$.
  By Lemma \ref{rrrrrrr}, $\mathbf{1}_B(y) \, \nu(dy)\text{-a.e.,} \ \pi_{Y=y}(I^c) > 0$ and $\pi_{Y=y}(A)<1$.
\end{proof}

\subsection{Proof of Theorem \ref{conter}}\label{sec:conter}
\begin{figure}[]
    \centering
    \includegraphics[width=0.8\textwidth]{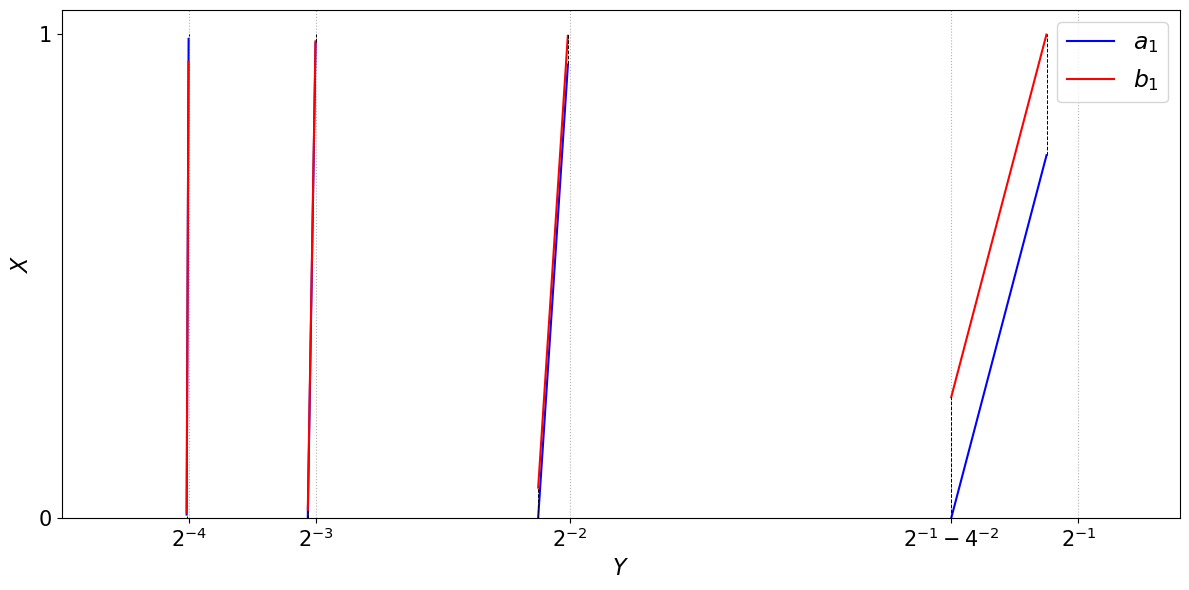} % Adjust width as needed
    \caption{Plot of functions \( a_n(y) \) and \( b_n(y) \) for \( n \) ranging from 1 to 4.}
    \label{fig:a_n_b_n2}
\end{figure}
We consider the law \( \pi \) on \( \mathbb{R}^2 \) defined by
\begin{align*}
\nu(dy) &= 15\, \mathbf{1}_{\bigcup_{m \in \mathbb{N}^*} \left[ 2^{-m} - 4^{-(m+1)},\, 2^{-m}- 4^{-(2m+1)} \right]}(y)\, dy
\end{align*}
\begin{align*}
&\nu(dy)\text{-a.e.,} \quad \pi_{Y = y}(dx) =
\frac{1}{2} \delta_{a(y)}(dx) + \frac{1}{2} \delta_{b(y)}(dx), 
\end{align*}
where
\[
a(y)
= \sum_{m \in \mathbb{N}^{*}}
a_m(y)\, \mathbf{1}_{\left[ 2^{-m} - 4^{-(m+1)},\, 2^{-m} - 4^{-(2m+1)} \right]}(y),
\quad
b(y)
= \sum_{m \in \mathbb{N}^{*}}
b_m(y)\, \mathbf{1}_{\left[ 2^{-m} - 4^{-(m+1)},\, 2^{-m} - 4^{-(2m+1)} \right]}(y)
\]
with
\[
\begin{cases}
a_m(y) = 4^{m+1}(y - 2^{-m}) + 1,\\[1em]
b_m(y) =  4^{m+1}\bigl(y - (2^{-m} - 4^{-(2m+1)})\bigr) + 1.
\end{cases}
\]
In particular, $a_m$ and $b_m$ are affine functions with the same slope and such that $b_m(y)-a_m(y)=4^{m+1} \times 4^{-(2m+1)} = {4^{-m}}$. The conditional laws \(\pi_{Y = y}\) are displayed as a function of \(y\) in Figure \ref{fig:a_n_b_n2}. By direct calculation, we check that
\begin{align}
\mu(dx)
&= f_{\mu}(x)\,dx
\label{eq:mu-density}
\\[1em]
f_\mu(x)
&= \frac{5}{4}\,\mathbf{1}_{\bigl[\tfrac14,\tfrac34\bigr]}(x)
  + \sum_{m \in \mathbb{N}^*}
    \frac{5}{8}\,(1+4^{-m})
    \,\mathbf{1}_{\bigl[4^{-(m+1)},\,4^{-m}\bigr]\cup\bigl[1-4^{-m},\,1-4^{-(m+1)}\bigr]}(x)
\label{eq:fmu-density}.
\end{align}
In particular, $\mu$ is equivalent to $\mathbf{1}_{[0,1]}\,\lambda$. Moreover, $\mu(dx)$-a.e., we have
\begin{align}
\pi_{X = x}(dy) &= \frac{1}{f_{\mu}(x)} 
\sum_{m\in\mathbb{N}^{*}} 4^{-(m+1)}
   \Bigl(
     \tfrac12\,\mathbf{1}_{[0, 1-4^{-(m+1)}]}(x)\,\delta_{a_m^{-1}(x)}(dy)
   + \tfrac12\,\mathbf{1}_{[4^{-(m+1)},1]}(x)\,\delta_{b_m^{-1}(x)}(dy)
   \Bigr).
\label{eq:pi-conditional}
\end{align}
\begin{rk}
The general idea of the proof is to handle the points of \(A\) in pairs. Indeed, for a given \(y_0 \in \bigcup_{m \in \mathbb{N}^*} \left[ 2^{-m} - 4^{-(m+1)},\, 2^{-m}- 4^{-(2m+1)} \right]\), the only \(x\)s such that 
$
\int_{\mathbb{R}} g(y)\,\pi_{X=x}(dy)
$
depends on \(g(y_0)\) are \(a(y_0)\) and \(b(y_0)\).
\end{rk}
 The proof of Theorem \ref{conter} relies on the following two results.
\begin{lemma}\label{mainlemma}
Let \( A  \in \mathcal{B}([0,1]) \). There exists a Borel set \(\widehat{A} \subset A\), a non-negative measurable function \(g_{\widehat{A}} \), and \(m\ge 3\) such that
\begin{itemize}\small
    \item[(i)] \(\lambda(\widehat{A}) \ge \frac{1}{80}\,\lambda(A),\)
    \item[(ii)] \(\mu(dx)\text{-a.e.,}\quad \int_{y \in \mathbb{R}} g_{\widehat{A}} (y)\,\pi_{X=x}(dy) = \mathbf{1}_{\widehat{A}}(x),\)
    \item[(iii)] \(\{y \in \mathbb{R} : g_{\widehat{A}} (y) > 0\} \subset \left[2^{-m} - 4^{-(m+1)},\, 2^{-m} - 4^{-(2m+1)}\right].\)
\end{itemize}
\end{lemma}
% \let\thefootnote\relax
% \footnote{We thank Anthony Quas (University of Victoria) for suggesting the use of Lebesgue's differentiation theorem. This allowed us to apply it successfully in Lemma~\ref{mainlemma} and in Corollary~\ref{cor2}.}
\begin{prop2} \label{inidcatrice}
   For \eqref{egtoutf} to hold it is enough that 
   $$\forall A\in{\mathcal B}(\R),\;\exists g:\R\to\R_+\mbox{ measurable such that }\mu(dx)\mbox{-a.e.},\quad \mathbf{1}_A(x)=\int_{y\in\R}g(y)\pi_{X=x}(dy).$$
\end{prop2}
\begin{proof}[Proof of Theorem \ref{conter}]
First, let us note that $D=\emptyset$.
Now let \( {A} \in \mathcal{B}([0,1]) \).  
Applying Lemma~\ref{mainlemma} to \( {A} \) gives us \( \widehat{A}_0\subset A \).  
We define \( A_1 = {A} \setminus \widehat{A}_0 \).  
Applying Lemma~\ref{mainlemma} to $A$ replaced by \( A_1 \) gives us \( \widehat{A}_1 \).  
We then define \( A_2 = A_1 \setminus \widehat{A}_1 \), and so on. We thus construct by induction a sequence \( (\widehat A_n)_{n \in \mathbb{N}} \) of disjoint Borel subsets of \( A \), and a sequence \( (g_{n})_{n \in \mathbb{N}} \) of non-negative measurable functions such that for all \( N \in \mathbb{N} \),
\[
\lambda\left(A \setminus \bigcup_{n=0}^{N} \widehat A_n\right) \leq \left(\frac{79}{80}\right)^N \lambda(A)\mbox{ and}
\]
\[
\mu(dx)\text{-a.e., }\int_{y \in \mathbb{R}} \sum_{n=0}^{N} g_{n}(y) \,\pi_{X=x}(dy) = \mathbf{1}_{\bigcup_{n=0}^{N} \widehat A_n}(x).
\]
Thus $\lambda(\bigcup_{n \in \mathbb{N}} \widehat A_n) =\lambda(A)$ and, since $\mu \ll \mathbf{1}_{[0,1]}\,\lambda$, $\mu(\bigcup_{n \in \mathbb{N}} \widehat A_n) =\mu(A)$. We finally get by the monotone convergence theorem that
\[
\mu(dx)\text{-a.e., } \int_{y \in \mathbb{R}} \sum_{n \in \mathbb{N}} g_{n}(y) \, \pi_{X = x}(dy) 
=\mathbf{1}_{\bigcup_{n \in \mathbb{N}} \widehat A_n}(x)  = \mathbf{1}_{A}(x).
\]
Finally, (\ref{egtoutf}) holds by Proposition \ref{inidcatrice}.
\end{proof}
For  $m \in \mathbb{N}^*$, we define  $\alpha_m = 2\times 4^{-m}$ and 
$I_m^k = [k \alpha_m, (k+1) \alpha_m)$ for \( k \in \mathbb{N} \).  The proof of Lemma \ref{mainlemma} relies on the following lemma.
\begin{lemma} \label{soussolution}
Let \( A \in \mathcal{B}([0,1]) \), $m\ge 3$ and 
$
k \in \{2,\dots, \frac{4^{m}}{2} - 3\}.
$
There exists a Borel set \( A_m^k \subset A \cap I_m^k \) and a non-negative measurable function \( g_{A_m^k} \) such that
\begin{itemize}\small
    \item[(i)] \(\mu(dx)\text{-a.e.,} \quad \int_{\mathbb{R}} g_{A_m^k}(y)\,\pi_{X=x}(dy) = \mathbf{1}_{A_m^k}(x),\)
    \item[(ii)] \(\{y \in \mathbb{R} : g_{A_m^k}(y)>0 \} \subset \left[2^{-m} - 4^{-(m+1)},\, 2^{-m} - 4^{-(2m+1)} \right],\)
    \item[(iii)] \(\lambda(A_m^k) \geq 2\lambda(A \cap I_m^k) - \alpha_m.\)
\end{itemize}
\end{lemma}
The Figure \ref{fig:a_n_b_n} below illustrates the idea of the proof of Lemma \ref{soussolution}. 
\begin{figure}[H]
    \centering
    \includegraphics[width=0.8\textwidth]{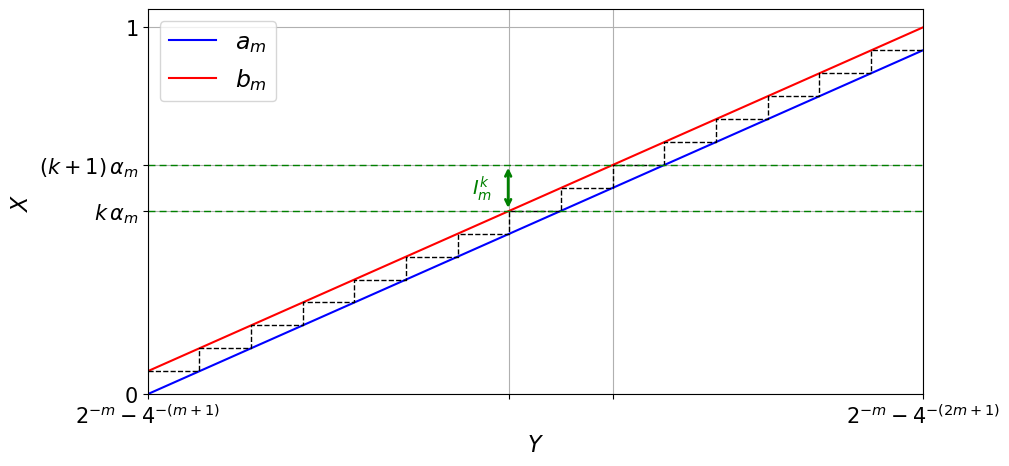} % Adjust width as needed
    \caption{Illustration of the proof of Lemma \ref{soussolution}}
    \label{fig:a_n_b_n}
\end{figure}
\begin{proof}[Proof]
Let \( A \in \mathcal{B}([0,1]) \), $m\ge 3$ and 
$
k \in \{2,\dots, \frac{4^{m}}{2} - 3\}.
$ We define the mesurable function $$I_m^k\ni x\mapsto R(x) =
\begin{cases}
x + \frac{\alpha_m}{2}, & \text{if } x \in [k\alpha_m, (k + \frac{1}{2}) \alpha_m), \\
x - \frac{\alpha_m}{2}, & \text{if } x \in [(k + \frac{1}{2}) \alpha_m, (k+1) \alpha_m),
\end{cases}
$$
and the Borel set
$$
A_m^k = R^{-1}\Bigl(A\cap I_m^k\Bigr) \cap \Bigl(A\cap I_m^k\Bigr)= R\Bigl(A\cap I_m^k\Bigr) \cap \Bigl(A\cap I_m^k\Bigr).
$$
(Note that the second equality holds because $R$ is an involution.) Since
$I_m^k \subset [2\alpha_m,1-2\alpha_m]$,
the functions \( a_m \) and \( b_m\) are affine functions with the same slope and such that \(b_m-a_m= \frac{\alpha_m}{2} \) on $a_{m}^{-1}\left([k\alpha_m, (k+\frac{1}{2})\alpha_m)\right)=b_{m}^{-1}([(k+\frac{1}{2}) \alpha_m, (k+1)\alpha_m))$. Thus
$$a_m^{-1} \left(A_m^k  \cap [k\alpha_m, (k+\tfrac{1}{2})\alpha_m)\right)
= b_m^{-1}\left(A_m^k \cap [(k+\tfrac{1}{2})\alpha_m, (k+1)\alpha_m)\right).$$
Since $2\alpha_m=4^{-(m-1)}\ge 4^{-\ell}$ for $\ell\ge m-1$ and $\frac{4^{-\ell}}{\alpha_m}=2\times 4^{m-1-\ell}$ is an integer for $1\le \ell\le m-2$, the density \(f_\mu\) given by (\ref{eq:fmu-density}) is constant on \(I_k^m\). With (\ref{eq:pi-conditional}), we deduce that
\[
\mathbf{1}_{\,a_m^{-1}\!\bigl(A_m^k \cap [\,k\alpha_m,\,(k+\tfrac12)\alpha_m)\bigr)}(y)\,\nu(dy)\text{-a.e., }
\pi_{X = a_m(y)}(\{y\})=
\pi_{X = b_m(y)}(\{y\}).
\]
We can thus define the measurable function
\[
g_{A_m^k}(y) = \frac{1}{\pi_{X = a_m(y)}(\{y\})} \, \mathbf{1}_{a_m^{-1}\left(A_m^k  \cap [k\alpha_m, (k+\tfrac{1}{2})\alpha_m)\right)}(y)
= \frac{1}{\pi_{X = b_m(y)}(\{y\})} \, \mathbf{1}_{b_m^{-1}\left(A_m^k \cap [(k+\tfrac{1}{2})\alpha_m, (k+1)\alpha_m)\right)}(y).
\]
Let $x \in A_m^k  \cap [k\alpha_m, (k+\frac{1}{2})\alpha_m).$
There exists a unique $y_x$ in $a_m^{-1}\left(A_m^k  \cap \left[k\alpha_m, (k+\frac{1}{2})\alpha_m\right)\right)$ such that $a_m(y_x) = x.$ Since \(\pi_{X = x}\) is a discrete measure with support in
$\bigcup_{\ell \in \mathbb{N}^*} (a_\ell^{-1}(\{x\}) \cup b_\ell^{-1}(\{x\}))$
it comes that
\[
\int_{y \in \mathbb{R}} g_{A_m^k}(y) \, \pi_{X = x}(dy)
=  g_{A_m^k }(y_x) \, \pi_{X = x}(\{y_x\})
=  g_{A_m^k }(y_x) \, \pi_{X = a_m(y_x)}(\{y_x\})
= 1.
\]
Let $x \in A_m^k  \cap [(k+\frac{1}{2})\alpha_m, (k+1)\alpha_m).$
There exists a unique $y_x$ in $b_m^{-1}\left(A_m^k  \cap [(k+\frac{1}{2})\alpha_m, (k+1)\alpha_m)\right)$ such that $b_m(y_x) = x.$ Since \(\pi_{X = x}\) is a discrete measure with support in \(\bigcup_{\ell \in \mathbb{N}^*} (a_\ell^{-1}(\{x\}) \cup b_\ell^{-1}(\{x\}))\), we then have 
\[
\int_{y \in \mathbb{R}} g_{A_m^k }(y) \, \pi_{X = x}(dy)
=  g_{A_m^k }(y_x) \, \pi_{X = x}(\{y_x\})
=  g_{A_m^k }(y_x) \, \pi_{X = b_m(y_x)}(\{y_x\})
= 1.
\]
Moreover, since for $y\in a_m^{-1}(A^k_m\cap [k\alpha_m, (k+\tfrac{1}{2})\alpha_m))$, we have $\{a_m(y),b_m(y)\}\subset A^k_m$, $g_{A^k_m}(y)\nu(dy)\text{-a.e.,}$ $\pi_{Y=y}(\{A^k_m\}^c)=0$ and
\begin{align*}
   \int_{x\in \{A^k_m\}^c}\int_{y\in\R} g_{A^k_m}(y)\,\pi_{X=x}(dy)\,\mu(dx)=\int_{y\in\R}g_{A^k_m}(y)\,\pi_{Y=y}(\{A^k_m\}^c)\,\nu(dy)=0.
\end{align*}
It follows that
\[
\mu(dx)\text{-a.e.,} \quad \int_{y \in \mathbb{R}} g_{A_m^k }(y) \, \pi_{X = x}(dy) = \mathbf{1}_{A_m^k }(x).
\]
We also have
\begin{align*}
\lambda(A_m^k)
&= \lambda\left((A\cap I_m^k) \cap R(A\cap I_m^k)\right) \\
&= \lambda(A\cap I_m^k) + \lambda \left(R(A\cap I_m^k)\right) - \lambda \left((A\cap I_m^k) \cup R(A\cap I_m^k)\right) \\
&\geq \lambda(A\cap I_m^k) + \lambda \left(R(A\cap I_m^k)\right) - \alpha_m \\
&= \lambda(A\cap I_m^k) + \lambda \left(R(A\cap [k\alpha_m, (k+\tfrac{1}{2})\alpha_m])\right)
+ \lambda \left(R(A\cap [(k+\tfrac{1}{2})\alpha_m, (k+1)\alpha_m])\right) - \alpha_m \\
&= \lambda(A\cap I_m^k) + 
\lambda \left(A\cap [(k+\tfrac{1}{2})\alpha_m, (k+1)\alpha_m]\right)
+ \lambda \left(A\cap [k\alpha_m, (k+\tfrac{1}{2})\alpha_m]\right) - \alpha_m \\
&= 2 \lambda(A\cap I_m^k) - \alpha_m.
\end{align*}
\end{proof}\begin{proof}[Proof of Lemma \ref{mainlemma}]
Let \( A  \in \mathcal{B}([0,1]) \). By the Lebesgue differentiation theorem (see \cite[page 456]{Jones2001}), we have
\[
\mathbf{1}_A(x)\, \lambda(dx) \text{-a.e., } \quad 
\frac{\lambda(A \cap [x-\varepsilon, x])}{\lambda([x-\varepsilon, x])}
\underset{\varepsilon \to 0^+}{\longrightarrow} 1,\quad
\frac{\lambda(A \cap [x, x+\varepsilon))}{\lambda([x, x+\varepsilon))}
\underset{\varepsilon \to 0^+}{\longrightarrow} 1.
\]
Let $\rho \in [\frac{63}{64}, 1)$. We thus have
\[
\mathbf{1}_A(x)\, \lambda(dx) \text{-a.e.,}\,\exists \varepsilon(x)>0,\, \forall \varepsilon' \in  (0, \varepsilon(x)],\quad
\min\left(
\frac{\lambda(A \cap [x-\varepsilon', x])}{\lambda([x-\varepsilon', x])},
\frac{\lambda(A \cap [x, x+\varepsilon'])}{\lambda([x, x+\varepsilon'])}
\right) > \rho.
\]
Since
$
\lambda(A \cap [2\eta', 1-2\eta']) \underset{\eta' \to 0^{+}}{\longrightarrow}  \lambda(A),
$
there exists $\eta_0 \in (0, \frac{1}{4})$ such that
\[
\forall \eta' \in (0, \eta_0],\quad \lambda\Bigl(A \cap [2\eta',1-2\eta']\Bigr) \ge \frac{1}{2}\lambda(A).
\]
Since
$
\lambda\Big(\{x \in A \cap [2\eta_0, 1-2\eta_0]: \varepsilon(x) \geq \eta'\}\Big) \underset{\eta' \to 0^{+}}{\longrightarrow}  \lambda(A \cap [2\eta_0, 1-2\eta_0]),
$
there exists \( \eta_1 > 0 \) such that
$$
\forall \eta' \in (0,\eta_1], \quad \lambda\Big(\{x \in A \cap [2\eta_0, 1-2\eta_0] : \varepsilon(x) \geq \eta'\}\Big) \geq \frac{\lambda(A \cap [2\eta_0, 1-2\eta_0])}{2}.
$$
Let $\eta=\min(\eta_0, \eta_1)$. We define
\begin{equation}
\tilde{A} = A \cap \bigcup_{\substack{x \in A \cap [2\eta,1-2\eta] \\ \varepsilon(x) \ge \eta}} (x-\eta, x+\eta).
\label{eq:tildeA_def}
\end{equation}
Since 
\[
\{x \in A \cap [2\eta,1-2\eta] : \varepsilon(x) \ge \eta\} \subset \tilde{A},
\]
it follows that
\begin{align}
\lambda(\tilde{A}) &\ge \lambda\Bigl(\{x \in A \cap [2\eta_0,1-2\eta_0] : \varepsilon(x) \ge \eta_1\}\Bigr) \nonumber\\[1mm]
&\ge \frac{1}{2}\lambda\Bigl(A \cap [2\eta_0,1-2\eta_0]\Bigr) \nonumber\\[1mm]
&\ge \frac{1}{4}\lambda(A) \label{Atild}.
\end{align}
According to the Vitali covering theorem (see \cite[page 448]{Jones2001}), there exists a subset $L$ of $\mathbb{N}$ and a sequence $(x_{\ell})_{{\ell} \in L}$ such that for all ${\ell} \in L$, 
$x_{\ell} \in A \cap [2\eta, 1-2\eta]$, $\varepsilon(x_{\ell}) \geq \eta,$
the intervals $(x_{\ell}-\eta, x_{\ell}+\eta)$ are pairwise disjoint, and
\begin{equation}
\bigcup_{\substack{x \in A \cap [2\eta, 1-2\eta] \\ \varepsilon(x) \geq \eta}} (x-\eta, x+\eta) 
\subset \bigcup_{{\ell} \in L} (x_{\ell}-5\eta, x_{\ell}+5\eta).
\label{vitalli}
\end{equation}
Observe that here
$
\operatorname{Card}(L) \le \left\lfloor \frac{1}{2\eta} \right\rfloor -1< +\infty.
$
Since $\frac{\eta}{2}\leq \frac{\eta_0}{2}<\frac 18=\alpha_2$, we can choose $m\ge 3$ such that $\frac{\eta}{2}\ge\alpha_m> \frac{\eta}{8}$. Note that if 
$
I_k^m\subset (x_\ell-\eta,\;x_\ell+\eta) \subset[\eta,\;1-\eta]\subset[2\alpha_m, 1-2\alpha_m],
$
then necessarily 
$
k \in \{2,\dots,\frac{4^{m}}{2}-3\}
$
and Lemma \ref{soussolution} applies. We define the following Borel set
\[
\widehat{A} = \bigcup_{\ell \in L, k \in \mathbb{N} \text{:} I_m^k \subset (x_{\ell}-\eta, x_{\ell}+\eta)} A_m^k,
\]
and the non-negative and mesurable function
\[
g_{\widehat{A}} = \sum_{\ell \in L, k \in \mathbb{N} \text{:} I_m^k \subset (x_{\ell}-\eta, x_{\ell}+\eta)} g_{A_m^k}.
\]
By the monotone convergence theorem, we then have
\begin{equation}
\mu(dx) \text{-a.e.,} \quad \int_{\mathbb{R}} g_{\widehat{A}}(y) \, \pi_{X=x}(dy)
= \sum_{\substack{\ell \in L, k \in \mathbb{N} \text{:} I_m^k \subset (x_{\ell}-\eta, x_{\ell}+\eta)}} \mathbf{1}_{A_m^k}(x)
= \mathbf{1}_{\bigcup_{\ell \in L, k \in \mathbb{N} \text{:} I_m^k \subset (x_{\ell}-\eta, x_{\ell}+\eta)} A_m^k}(x)
= \mathbf{1}_{\widehat{A}}(x).
\label{eq:ghatA_projection}
\end{equation}
This proves assertion (ii). Assertion (iii) holds by the definition of the functions \( g_{A_m^k} \). To complete the proof, let us now prove (i).
\begin{comment}
To conclude. \julien{Je te laisse re-rédiger ça -- pour que tu apprennes à bien rédiger. Déjà, tu confonds (i) et (ii). Ensuite, (ii) est exactement (2.9). (ii) et (iii) ont déjà été prouvés, et il faut dire que (ii) a été prouvé dès qu'il a été prouvé, et pareil pour (iii). Et là tu dis, "to complete the proof, let us now prove (i)".}
\end{comment}
 Let $\ell \in L$ and $k \in \mathbb{N} \text{ be such that } I_m^k \subset (x_{\ell}-\eta, x_{\ell}+\eta)$. We have
\begin{align}
\lambda(A \cap I_m^k)
&= \lambda(A \cap [x_{\ell} - \eta,\, x_{\ell} + \eta])
 - \lambda\left(A \cap \left\{[x_{\ell} - \eta,\, x_{\ell} + \eta] \setminus I_m^k\right\}\right) \nonumber\\
&\geq \lambda(A \cap [x_{\ell}-\eta, x_{\ell}+\eta]) - \lambda([x_{\ell}-\eta, x_{\ell}+\eta] \setminus I_m^k) \nonumber\\[1mm]
&= \lambda(A \cap [x_{\ell}-\eta, x_{\ell}+\eta]) - (2\eta - \alpha_m) \nonumber\\[1mm]
&= 2 \eta \times \frac{\lambda(A \cap [x_{\ell}-\eta, x_{\ell}+\eta])}{\lambda([x_{\ell}-\eta, x_{\ell}+\eta])} - (2\eta - \alpha_m) \nonumber\\[1mm]
&= 2 \eta \times \frac{1}{2} \left( \frac{\lambda(A \cap [x_{\ell}-\eta, x_{\ell}])}{\lambda([x_{\ell}-\eta, x_{\ell}])} + \frac{\lambda(A \cap [x_{\ell}, x_{\ell}+\eta])}{\lambda([x_{\ell}, x_{\ell}+\eta])} \right) - (2\eta - \alpha_m) \nonumber\\[1mm]
&\geq 2 \eta \rho - (2\eta - \alpha_m), \quad \text{since } x_{\ell} \in \{ x \in A \cap [\eta, 1-\eta] : \varepsilon(x) \geq \eta \} \nonumber\\[1mm]
&= -2\eta(1-\rho)+ \alpha_m \nonumber\\[1mm]
&\geq \alpha_m (16\rho - 15), \quad \text{since }  8 \alpha_m \geq \eta.\label{lower_bound_lambda}
\end{align}
With Lemma \ref{soussolution} (iii) and using that $\rho \geq \frac{63}{64}$, we deduce that
\begin{align}
\lambda(A_m^k) \geq 2 \lambda(A \cap I_m^k) - \alpha_m \geq \alpha_m (32\rho - 31) \geq \frac{\alpha_m}{2}. \label{I_nk}
\end{align}
For $\ell \in L$, using \eqref{I_nk} for the first inequality, then that $(x_{\ell}-\eta, x_{\ell}+\eta)$ contains at least $\left\lfloor \frac{2\eta}{\alpha_m} \right\rfloor -1$ intervals $I^k_m$ for the second inequality, and  that $\alpha_m \leq \eta/2$ for the last one, we obtain
\begin{align}
\lambda \left( \bigcup_{k \in \mathbb{N} \text{:} I_m^k \subset (x_{\ell}-\eta, x_{\ell}+\eta)} A_m^k \right)
&= \sum_{k \in \mathbb{N} \text{:} I_m^k \subset (x_{\ell}-\eta, x_{\ell}+\eta)} \lambda(A_m^k) \nonumber \\
&\geq \sum_{k \in \mathbb{N} \text{:} I_m^k \subset (x_{\ell}-\eta, x_{\ell}+\eta)} \frac{\alpha_m}{2} 
\quad \nonumber \\
&\geq \left( \left\lfloor \frac{2\eta}{\alpha_m} \right\rfloor -1 \right) \frac{\alpha_m}{2} \nonumber \\
&\ge  \eta-\alpha_m  \nonumber\\
&= \frac{1}{2} \left( 1 - \frac{\alpha_m}{\eta} \right) \lambda\left((x_{\ell}-\eta, x_{\ell}+\eta)\right) \nonumber \\
&\geq \frac{1}{4} \lambda \left((x_{\ell}-\eta, x_{\ell}+\eta)\right). \label{UnionA_k}
\end{align}
Thus
\begin{align*}
\lambda(\widehat{A}) &=\sum_{\ell \in L} \lambda \left( \bigcup_{k \in \mathbb{N} \text{:} I_m^k \subset (x_{\ell}-\eta, x_{\ell}+\eta)} A_m^k \right) \\[1mm]
&\geq \sum_{\ell \in L} \frac{1}{4} \lambda\left((x_{\ell}-\eta, x_{\ell}+\eta)\right), \quad \text{by (\ref{UnionA_k}),} \\[1mm]
&\geq \sum_{\ell \in L} \frac{1}{20} \lambda\left((x_{\ell}-5\eta, x_{\ell}+5\eta)\right) \\[1mm]
&\geq \frac{1}{20} \lambda \left( \bigcup_{\ell \in L} (x_{\ell}-5\eta, x_{\ell}+5\eta) \right) \\[1mm]
&\geq \frac{1}{20} \lambda \left( \bigcup_{\substack{x \in A \cap [2\eta,1-2\eta]\\\varepsilon(x) \geq \eta}} (x_{\ell} - \eta, x_{\ell} + \eta) \right), \quad \text{ by (\ref{vitalli})}, \\[1mm]
&\geq \frac{1}{20} \lambda(\tilde{A}), \quad \text{ by (\ref{eq:tildeA_def}),}\\[1mm]
&\geq \frac{1}{80} \lambda(A), \quad \text{ by (\ref{Atild}),}
\end{align*}
which shows (ii). 
\end{proof}

\begin{proof}[Proof of Proposition \ref{inidcatrice}]
In the proof, a function \( f : \mathbb{R} \to \mathbb{R} \) will be called a simple function if it can be expressed as
\[
f(x) = \sum_{i=1}^n a_i \mathbf{1}_{A_i}(x),
\]
where $n \in \mathbb{N}$, \( (a_i)_{i \in \{1, \ldots, n\}} \in \mathbb{R}^n \), and \( (A_i)_{i \in \{1, \ldots, n\}} \) is a family of pairwise disjoint Borel sets.
Let $(X,Y)$ be a random couple distributed according to $\pi$.
Suppose that for every \( A \in \mathcal{B}(\mathbb{R})\), there exists a non-negative measurable function \( g_A \) such that \( \mathbb{E}[g_A(Y) | X] = \mathbf{1}_A(X) \). By linearity of the conditional expectation, for every simple function \( f \in L^1_+(\mu) \), there exists \( g \in L^1_+(\nu) \) such that \( \mathbb{E}[g(Y) | X] = f(X) \). Let \( f \in L^1_+(\mu) \). There exists a non-decreasing sequence \( (f_n)_{n \in \mathbb{N}} \) of non-negative simple functions in \( L^1_+(\mu) \) such that \( f_n \xrightarrow{L^1(\mu)} f \). Let us construct by induction an non-decreasing sequence \( (g_{n})_{n \in \mathbb{N}} \) of functions in \( L^1_+(\nu) \) such that
\begin{equation}
\forall n \in \mathbb{N}, \quad \mathbb{E}[g_{n}(Y) | X] = f_n(X)\text{ a.s..}
\label{eq:conditional_expectation}
\end{equation}
Let \( g_{0} \in L^1_+(\nu) \) such that \( \mathbb{E}[g_{0}(Y) | X] = f_0(X) \). Since \( f_{n+1} - f_n \) is a non-negative simple function in \( L^1_+(\mu) \), there exists \( g_{f_{n+1} - f_n} \in L^1_+(\nu) \) such that a.s., 
\[
\mathbb{E}[g_{f_{n+1} - f_n}(Y) | X] = (f_{n+1} - f_n)(X)\text{}.
\]
Set \( g_{n+1} = g_{f_{n+1} - f_n} + g_{n} \).
By linearity of the conditional expectation, we have a.s.,
\[
\mathbb{E}[g_{n+1}(Y) | X] = (f_{n+1} - f_n)(X) + f_n(X) = f_{n+1}(X).
\]
Moreover, \(\nu(dy)\)-a.e., \(g_{n+1} \geq g_{n}\), because \(g_{f_{n+1} - f_n} \geq 0\).
Observe that for all $p \geq q$,
\begin{align*}
\|g_{p} - g_{q}\|_{L^1(\nu)} 
&= \mathbb{E}[|g_{p}(Y) - g_{q}(Y)|] \\
&= \mathbb{E}[g_{p}(Y) - g_{q}(Y)] \\
&= \mathbb{E}\left[\mathbb{E}[g_{p}(Y) | X] - \mathbb{E}[g_{q}(Y) | X]\right] \\
&= \mathbb{E}[f_p(X) - f_q(X)] \\
&= \mathbb{E}[|f_p(X) - f_q(X)|] \\
&= \|f_p - f_q\|_{L^1(\mu)}.
\end{align*}
Thus, \((g_{n})_{n \in \mathbb{N}}\) is a Cauchy sequence in the closed subset \(L^1_+(\nu)\) of \(L^1(\nu)\). Since \(L^1(\nu)\) is complete, there exists a function \(g \in L^1_+(\nu)\) such that
$
g_{n} \xrightarrow{L^1(\nu)} g.
$
Since \( f_n \xrightarrow{L^1(\mu)} f \), it comes from \eqref{eq:conditional_expectation} and the continuity of the operator $T$ defined by \eqref{eq:def_T_operator} that 
$
\mathbb{E}[g(Y) | X] = f(X)\text{ a.s..}
$
\end{proof}
\section{Surjectivity without sign constraint}
In this section, we relax the non-negativity constraint in order to understand which conditions obtained in the previous section stem from non-negativity, and which ones arise from surjectivity. We are thus interested in probability measures \(\pi\) such that
\begin{equation*}
   \forall f \in L^1(\mu),\; \exists g \in L^1(\nu) \text{ measurable such that } \mu(dx) \text{-a.e.},\quad 
   f(x) = \int_{y \in \mathcal{Y}} g(y) \pi_{X=x}(dy). \tag{$\mathcal{R}$} \label{egtoutf2}
\end{equation*}
We recall that this amounts to studying the surjectivity of the operator
\[
T : L^1(\nu) \to L^1(\mu),
\]
defined as
\[
Tg(x) = \int_{y \in \mathcal{Y}} g(y) \, \pi_{X=x}(dy).
\]
We easily verify that \(T\) is linear and bounded, since
\[
\forall g \in L^1(\nu), \quad \|Tg\|_{L^1(\mu)} \leq \|g\|_{L^1(\nu)}.
\]
The study of surjectivity, without non-negativity constraints, between two Banach spaces can be carried out using existing results from functional analysis, related to duality. This functional analysis approach also allows us to better understand the necessary condition in Proposition \ref{propcn}, which was used to prove Theorem \ref{mainsuffisant}. In the rest of the section, \( \mathring{B}_{L^1(\nu)} \) and \( \mathring{B}_{L^1(\mu)} \) denote respectively the unit open balls of \( L^1(\nu) \) and \( L^1(\mu) \), and $\overline{B}_{L^1(\mu)}$, $\overline{B}_{L^1(\mu)}$ are respectively the unit closed balls of $L^1(\nu)$ and $L^1(\mu)$.
\begin{rk}
Proposition \ref{propcn} implies in particular that if \eqref{egtoutf} holds, then 
$$\forall A \in \mathcal{B}(\mathbb{R})\text{ such that } \mu(A) > 0,\quad \nu\bigl(\{y \in \mathcal{Y} : \pi_{Y=y}(A) = 1\}\bigr) > 0.$$ 
\end{rk}
The next two propositions show that the necessary condition in Proposition \ref{propcn} and (\ref{egtoutf}) are in fact very close: the necessary condition implies that $T(\mathring{B}_{L^1(\nu)}) =\mathring{B}_{L^1(\mu)}$, while $\eqref{egtoutf}$ is equivalent to $T(\overline{B}_{L^1(\nu)}) = \overline{B}_{L^1(\mu)}$. Moreover, under this necessary condition, \eqref{egtoutf2} holds by (iii) of Proposition \ref{corball}. (Indeed, let \( f \in L^1(\mu) \). There exists \( g \in \mathring{B}_{L^1(\nu)} \) such that
$
Tg = \frac{f}{\|f\|_{L^1(\mu)} + 1}.
$
Therefore,
$T\left( (\|f\|_{L^1(\mu)} + 1) g \right) = f.$)

\begin{comment}
\julien{cette dernière assertion est-elle évidente ?}
\end{comment}
\begin{prop2} \label{corball}
The assertion 
\begin{itemize}
\item[(i)] $\forall A \in \mathcal{B}(\mathbb{R}),\text{ such that } \mu(A) > 0,\quad \nu\bigl(\{y \in \mathcal{Y}: \pi_{Y=y}(A) = 1\}\bigr) > 0$
\end{itemize}
implies the following equivalent assertions:
\begin{itemize}
    \item[(ii)]  $\forall A \in \mathcal{B}(\mathbb{R})$ such that $\mu(A) > 0$, \quad 
    $
    \operatorname{ess\,sup}_{\nu(dy)} \pi_{Y=y}(A) = 1,
    $  
      
    \item[(iii)] $T(\mathring{B}_{L^1(\nu)}) =\mathring{B}_{L^1(\mu)}$.
\end{itemize}
\end{prop2}
The proof is postponed to the end of this section.
\begin{prop2}
The following assertions are equivalent to \eqref{egtoutf}:
\begin{itemize}
\item[(i)] $\forall f \in L^1(\mu),\; \exists g \in L^1(\nu) \text{ such that } Tg = f \text{ and } \|g\|_{L^1(\nu)} = \|f\|_{L^1(\mu)},$
\item[(ii)] $T(\overline{B}_{L^1(\nu)}) = \overline{B}_{L^1(\mu)}$.
\end{itemize}
\end{prop2}
\begin{proof}
Let us first show that \eqref{egtoutf} $\implies$ (i). Suppose that (\ref{egtoutf}) holds.
Let $f^+$ and $f^-$ denote the positive and negative parts of $f$ so that $f=f^{+}-f^{-}$. Let $g_1, g_2\in L^1_+(\nu)$ be such that $Tg_1=f^{+}$ and $Tg_2=f^{-}$. For $g=g_1-g_2$, we have $Tg=f$. Using that $T$ is isometric on $L^{1}_+(\nu)$ and non-expansive on $L^1(\nu)$, we get
\begin{align*}
\|f\|_{L^1(\mu)}=\|f^+\|_{L^1(\mu)}+\|f^-\|_{L^1(\mu)}=\|g_1\|_{L^1(\nu)}+\|g_2\|_{L^1(\nu)}\ge \|g\|_{L^1(\nu)}\ge\|Tg\|_{L^1(\mu)}=\|f\|_{L^1(\mu)}% \|f\|_{L^{1}(\mu)},
%   &= \|Tg\|_{L^{1}(\mu)}                                               \\[2pt]
%   &= \bigl\|\,T\bigl(g^{+}-g^{-}\bigr)\bigr\|_{L^{1}(\mu)}              \\[2pt]
%   &\le \|Tg^{+}\|_{L^{1}(\mu)} + \|Tg^{-}\|_{L^{1}(\mu)}               \\[2pt]
%   &= \|g^{+}\|_{L^{1}(\nu)} + \|g^{-}\|_{L^{1}(\nu)}              \\[2pt]
%   &= \|f^{+}\|_{L^{1}(\mu)} + \|f^{-}\|_{L^{1}(\mu)}\\
% &=\|f\|_{L^{1}(\mu)},
\end{align*}
so that
$
\|g\|_{L^{1}(\nu)}=\|f\|_{L^{1}(\mu)}
$.

Since $T$ is non-expansive, $T(\overline{B}_{L^1(\nu)})\subset \overline{B}_{L^1(\mu)}$. Under (i), the reverse inclusion also holds, so (i) implies (ii).

Let us finally show that (ii) $\implies$ \eqref{egtoutf}. Let $f \in L^1_+(\mu)$ be non-zero. Up to replacing $f$ by $\frac{f}{\|f\|_{L^1(\mu)}}$, we may suppose that $\|f\|_{L^1(\mu)} = 1$. There exists $g \in \overline{B}_{L^1(\nu)}$ such that $Tg = f$. Since $T$ is non-expansive, we have $\|g\|_{L^1(\nu)} \geq \|Tg\|_{L^1(\mu)} = 1$, and since $g \in \overline{B}_{L^1(\nu)}$, it follows that $\|g\|_{L^1(\nu)} = 1= \|f\|_{L^1(\mu)}$. Moreover, by integrating the equality $Tg(x)=f(x)$ against $\mu(dx)$, we get
\[
\int_{x \in \mathbb{R}} f(x)\, \mu(dx) = \int_{y \in \mathcal{Y}} g(y)\, \nu(dy).
\]
It follows that
\[
\int_{y \in \mathcal{Y}} (\left| g(y)\right|  - g(y)) \, \nu(dy) 
= \int_{x \in \mathbb{R}} (\left| f(x)\right| - f(x)) \, \mu(dx) = 0.
\]
Thus, \,\(\nu(dy)\)-a.e., \( g(y) \geq 0 \), and it follows that $
g \in L^1_+(\nu).
$
\end{proof}
\begin{rk}
Since $T$ is non-expansive, $T(\overline{B}_{L^1(\nu)})\subset \overline{B}_{L^1(\mu)}$ with equality when  $T(\mathring{B}_{L^1(\nu)}) =\mathring{B}_{L^1(\mu)}$ and $T(\overline{B}_{L^1(\nu)})$ is closed. In view of Proposition \ref{corball}, we deduce that the necessary condition 
\[
\forall A \in \mathcal{B}(\mathbb{R})\mbox{ such that } \mu(A) > 0,\; \nu\bigl(\{y \in \mathcal{Y} : \pi_{Y=y}(A) = 1\}\bigr) > 0
\]is equivalent to (\ref{egtoutf}) if and only if $T(\overline{B}_{L^1(\nu)})$ is closed. At this stage of our research, we do not know whether the two statements are equivalent in general.
\end{rk}
\begin{rk}
The assertion (i) of Proposition \ref{corball} is not equivalent to (ii) or (iii). Indeed, to prove Proposition \ref{cor2} below, we construct a law $\pi$ such that (i) does not hold by Theorem \ref{mainsuffisant}, since $\nu(dy)$-a.e., the cumulative distribution function of $\pi_{Y=y}$ is increasing on an interval of positive length, and (ii) holds, since $\delta$ can be chosen arbitrarily close to $1$. In particular, the assertions $T(\mathring{B}_{L^1(\nu)}) =\mathring{B}_{L^1(\mu)}$ and $T(\overline{B}_{L^1(\nu)}) = \overline{B}_{L^1(\mu)}$ are not equivalent.
\end{rk}
\begin{example}[A simple example where \eqref{egtoutf} does not hold and $T$ is surjective] \label{remark24}
We present here an easy example, much simpler than the one constructed in Theorem \ref{conter}, where \(\nu(D) = 0\)$,\,\nu$ and $\mu$ are absolutely continuous with respect to the Lebesgue measure, and $T$ is surjective but (\ref{egtoutf})
does not hold. Let \(p \in (0, 1)\) and \(I \sim \text{Bernoulli}(p)\) be independent of two independent random variables \(X\) and \(X'\) with  common law \(\mu\). We define
\[
Y = I X + (1-I) X'.
\]
Note that \(Y\) has the same distribution as \(X\), i.e., $\mu=\nu$. The joint distribution of $(X,Y,I)$ thus reads
\[
\gamma(dx,dy,di)=p \, \delta_1(di) \, \mu(dx) \, \delta_x(dy) +(1-p) \, \delta_0(di) \, \mu(dx) \, \mu(dy).
\]
By integrating over \(i\), we get that the joint distribution of $(X,Y)$ is:  
\[
\pi(dx,dy) = \mu(dx) \left(p \, \delta_x(dy) + (1-p) \, \mu(dy)\right)
= \mu(dy) \big(p\, \delta_y(dx) + (1-p)\, \mu(dx)\big).
\]
Let \(f \in L^1(\mu)\). Observe that we have \(\mu(dx)\)-a.e.,
\[
\pi_{X=x}(dy) = p\,\delta_{x}(dy) + (1-p)\, \mu(dy).  
\]
Let \(g = \frac{f - \mathbb{E}[f(X)]}{p} + \mathbb{E}[f(X)]\). Note that $g \in L^1(\nu)$ and $\mathbb{E}[f(X)] = \mathbb{E}[g(Y)]$. Moreover, we have \(\mu(dx)\)-a.e.,
\begin{align*}
\int_{y \in \mathbb{R}} g(y)\, \pi_{X=x}(dy) &= \int_{y \in \mathbb{R}} g(y) \left( p\, \delta_{x}(dy) + (1-p)\, \mu(dy) \right) \\
&= p \int_{y \in \mathbb{R}} g(y) \, \delta_{x}(dy) + (1-p) \int_{y \in \mathbb{R}} g(y) \, \mu(dy) \\
&= p \,g(x) + (1-p)\, \mathbb{E}[g(X)]\\
&= p \,g(x) + (1-p)\, \mathbb{E}[f(X)]\\
&= f(x).
\end{align*}
Therefore, $\eqref{egtoutf2}$ holds. When \(\mu\) is not a Dirac mass, there exists \(A \in \mathcal{B}(\mathbb{R})\) such that \(0 < \mu(A) < 1\). Let us consider \(f = \mathbf{1}_A\), which belongs to \(L^1_+(\mu)\). Then
$
\mathbb{E}[f(Y)] = \mu(A),$ and $\{y \in \mathbb{R} : g(y) < 0\} = \{y \in \mathbb{R} : f(y) < (1-p)\,\mathbb{E}[f(Y)]\} = A^c.
$
Since \(\nu(A^c) = \mu(A^c) > 0\), we have \(\nu(\{y \in \mathbb{R} : g(y) < 0\}) > 0\). Thus, \(g \notin L^1_+(\nu)\). It is therefore natural to ask whether \eqref{egtoutf} holds when \(\mu\) is not a Dirac mass.
In fact,
since \(\nu(dy)\)-a.e.,
\(
\pi_{Y=y}(dx) = p\, \delta_y(dx) + (1-p)\, \mu(dx),
\)
we have $\mu(D)=0$ and \(\nu(dy)\)\text{-a.e., } \(\mu \ll \pi_{Y=y}\). Either $\mu$ is discrete and (\ref{egtoutf})
does not hold by Corollary \ref{defC} or $\mu$ has a non-zero diffuse component and (\ref{egtoutf})
does not hold by Corollary \ref{Lebesguedensity}.
\end{example}
Let \(U\) and \(V\) be two Banach spaces and \(L\) be a bounded linear operator from \(U\) to \(V\). Let \(U^*\) and \(V^*\) denote the topological duals of \(U\) and \(V\). The notation \(\langle \cdot, \cdot \rangle_U\) (resp. \(\langle \cdot, \cdot \rangle_V\)) denotes the duality pairing, corresponding to the bilinear application of \(U^* \times U\) (resp. \(V^* \times V\)) that associates a scalar to a pair formed by a linear form of \(U^*\) (resp., \(V^*\)) and a vector of \(U\) (resp., \(V\)). To characterize the surjectivity of the operator $L$, it is useful to introduce its adjoint $L^*$.
\begin{def2}\label{definition6}
The adjoint operator \(L^*\) of \(L\) is the unique linear bounded operator from \(V^*\) to \(U^*\) satisfying the equality
\[
\forall u \in U,\, \forall v^* \in V^*, \quad \langle v^*, Lu \rangle_V = \langle L^* v^*, u \rangle_U.
\]
\end{def2}
\begin{thm2}\cite[Theorem 4.13, page 100]{Rudin1991}\label{corollary1}
The operator \(L\) is surjective if and only if there exists \(\delta > 0\) such that
\[
\forall v^* \in V^*, \quad \|L^*v^*\|_{U^*} \geq \delta \|v^*\|_{V^*}.
\]
\end{thm2}
%See Theorem 4.13 of \cite[page 100]{Rudin1991}. 
Let us compute the adjoint $T^*$ of $T$.
\begin{prop2} \label{proposition14}
The operator \(T^*\) defined as
\[ T^* : L^\infty(\mu) \to L^\infty(\nu) \]
\[ T^*f(y) = \int_{x \in \mathbb{R}} f(x) \, \pi_{Y=y} (dx)= \mathbb{E}[f(X)|Y=y] \]
is the adjoint of \(T\).
\end{prop2}
\begin{proof}
Here \(U = L^1(\nu)\) and \(V = L^1(\mu)\), \(U^* = L^\infty(\nu)\) and \(V^* = L^\infty(\mu)\).
The linearity is clear. Let us first show that the operator \( T^* \) is well defined and bounded from \( L^{\infty}(\mu) \) into \( L^{\infty}(\nu) \).
Indeed, let \( f^* \in L^\infty(\mu) \). We have \( \mu(dx) \)-a.e.,
$
|f^*(x)| \leq \|f^*\|_{L^\infty(\mu)},
$
and \( \nu(dy) \)-a.e., \( \operatorname{supp}(\pi_{Y=y}) \subset \operatorname{supp}(\mu) \), so that
\begin{align*}
\|T^*f^*\|_{L^\infty(\nu)} 
&= \int_{\operatorname{supp}(\pi_{Y=y})} f^*(x) \, \pi_{Y=y}(dx) \\
&\leq \int_{\operatorname{supp}(\mu)} \|f^*\|_{L^\infty(\mu)} \, \pi_{Y=y}(dx) \\
&= \|f^*\|_{L^\infty(\mu)} < +\infty.
\end{align*}
Let us check that \( T^* \) is the adjoint of \( T \). For all \(f^* \in L^\infty(\mu)\) and \(g \in L^1(\mu)\),
\begin{align*}
\int_{x \in \mathbb{R}} f^*(x)\, T g(x) \, \mu(dx) &= \int_{x \in \mathbb{R}} f^*(x) \left(\int_{y \in \mathcal{Y}} g(y) \, \pi_{X=x} (dy)\right) \, \mu(dx) \\
    &= \int_{y \in \mathcal{Y}} g(y) \left(\int_{x \in \mathbb{R}} f^*(x) \, \pi_{Y=y}(dx)\right) \, \nu(dy) \\
    &= \int_{y \in \mathcal{Y}} g(y)\, T^*f^*(y) \, \nu(dy).
\end{align*}
\end{proof}
The following theorem provides a criterion to verify the surjectivity of $T$ in a non-constructive way.
\begin{thm2}\label{corollary2}
The following statements are equivalent:
\begin{itemize}
    \item[(i)] \(T\) is surjective.
    \item[(ii)] There exists \(\delta > 0\) such that for every function \(f^* \in L^\infty(\mu)\),
    \[
    \|T^*f^*\|_{L^\infty(\nu)} \geq \delta\, \|f^*\|_{L^\infty(\mu)}.
    \]
    \item[(iii)] There exists \(\delta > 0\) such that for every simple function \(f^*\),
    \[
    \|T^*f^*\|_{L^\infty(\nu)} \geq \delta\, \|f^*\|_{L^\infty(\mu)}.
    \]
\end{itemize}
\end{thm2}
The proof of Theorem \ref{corollary2} relies on Theorem \ref{corollary1} and the following result.
\begin{lemma}\cite[page 35]{Stein2011}\label{lemma_simple_functions_dense}
Let  $\rho$ be a non-negative finite measure on $(\mathbb{R}, \mathcal{B}(\mathbb{R}))$. The simple functions are dense in \(L^\infty(\rho)\).
\end{lemma}
\begin{proof}[Proof of Theorem \ref{corollary2}]
By Theorem \ref{corollary1}, (i) \(\Longleftrightarrow \) (ii). Moreover, (ii) \(\implies\) (iii) is obvious. 
Let us show (iii) \(\implies\) (ii). Let \( f^* \in L^\infty(\mu) \). By density of the simple functions in \( L^\infty(\mu) \) (see Lemma \ref{lemma_simple_functions_dense}), there exists a sequence of simple functions \( (f^*_n)_{n \in \mathbb{N}} \) in \( L^\infty(\mu) \) such that \( f^*_n \xrightarrow{L^\infty(\mu)} f^* \). Since \( T^* \) is a bounded linear operator, it is continuous. By continuity, \( T^*f^*_n \xrightarrow{L^\infty(\nu)} T^*f^* \). Taking the limit in the inequality \( \|T^*f^*_n\|_{L^\infty(\nu)} \geq \delta \|f^*_n\|_{L^\infty(\mu)} \), we have $
\|T^*f^*\|_{L^\infty(\nu)} \geq \delta \|f^*\|_{L^\infty(\mu)}.
$
\end{proof}
\begin{rk}
According to Theorem~\ref{corollary2},  
if \(T\) is surjective, then there exists \(\delta>0\) such that for every
\(A\in\mathcal{B}(\mathbb{R})\) with \(\mu(A)>0\),\,
$
\operatorname*{ess\,sup}_{\nu(dy)} \pi_{Y=y}(A) \;>\; \delta .
$
Hence,
\[
\exists\,\delta'>0,\, \forall\,A\in\mathcal{B}(\mathbb{R})\text{ such that } \mu(A)>0, \quad \nu\bigl(\{\,y\in\mathcal{Y} : \pi_{Y=y}(A)>\delta'\}\bigr)\;>\;0 .
\]
This is  a weakened form of the necessary condition for \eqref{egtoutf} to hold:
\[
\forall\,A\in\mathcal{B}(\mathbb{R})\text{ such that } \mu(A)>0,\quad
\nu\bigl(\{y \in \mathcal{Y} :  \pi_{Y=y}(A)=1\}\bigr) \;>\; 0 .
\]
\end{rk}

We deduce from Theorem \ref{corollary2} the following surjectivity criterion on \( \pi \).
\begin{prop2} \label{kkk}
If 
\(
\xi :=\inf_{\substack{A \in \mathcal{B}(\mathbb{R}) : \mu(A) > 0}}\, \operatorname{ess\,sup}_{\nu(dy)} \pi_{Y=y}(A) > \frac{1}{2},
\)
then (iii) in Theorem \ref{corollary2} holds with $\delta=2\xi-1$ and  \eqref{egtoutf2} holds.
\end{prop2}
\begin{rk}
The sufficient condition given in Proposition \ref{kkk} is not necessary. Indeed, in Example \ref{remark24}, for \(A \in \mathcal{B}(\mathbb{R})\) such that \(\mu(A) > 0\),
\[
\operatorname{ess\,sup}_{\nu(dy)} \pi_{Y=y}(A) = p + (1-p) \,\mu(A).
\]
Suppose in this example that \(\mu\) is a measure that cannot be written in the form \(\frac{1}{2} \delta_{x_1} + \frac{1}{2} \delta_{x_2}\), where \(x_1\) and \(x_2\) are real numbers, possibly equal. Then, there exists a Borel set \(A_0 \in \mathcal{B}(\mathbb{R})\) such that \(0 < \mu(A_0) < \frac{1}{2}\).

In the case when \(p < \frac{\frac{1}{2} - \mu(A_0)}{1 - \mu(A_0)}\),
\[
\operatorname{ess\,sup}_{\nu(dy)} \pi_{Y=y}(A_0) = p + (1-p)\, \mu(A_0) = \mu(A_0) + p\,(1 - \mu(A_0)).
\]
\[
< \mu(A_0) + \frac{1}{2} - \mu(A_0) < \frac{1}{2}.
\]
Thus,
\[
\inf_{\substack{A \in \mathcal{B}(\mathbb{R}) : \mu(A) > 0}}\, \operatorname{ess\,sup}_{\nu(dy)} \pi_{Y=y}(A) < \frac{1}{2}.
\]
\end{rk}\begin{proof}
Let \( f^* \) be a simple function in \( L^\infty(\mu) \).  
There exists a finite index set \( I \), a collection of disjoint Borel sets \( (A_i)_{i \in I} \) with \( \mu(A_i) > 0 \),  
and a vector \( (f^*_i)_{i \in I}\in\R^I\) such that
\[
\forall x \in \mathbb{R},\text{ }f^*(x) = \sum_{i \in I} f^*_i\, \mathbf{1}_{A_i}(x).
\]
Let \( i_0 \in I \) such that \(\max_{i \in I} \lvert f_i^* \rvert = \lvert f^*_{i_0} \rvert = \|f^*\|_{L^\infty(\mu)}\).
We have \(\nu(dy)\)-a.e.,
\begin{align*}
\left| T^*f^*(y) \right| &= \left| \sum_{i \in I} f^*_i \, \pi_{Y=y}(A_i) \right| \\
&\geq \left| f^*_{i_0} \right| \pi_{Y=y}(A_{i_0}) - \sum_{i \in I \setminus \{i_0\}} \left| f^*_{i_0} \right| \pi_{Y=y}(A_i) \\
&= \left( 2\, \pi_{Y=y}(A_{i_0}) - 1 \right) \left| f^*_{i_0} \right|.
\end{align*}
Hence,
\begin{align*} \label{essup}
\|T^*f^*\|_{L^\infty(\nu)} &\geq \left(2 \operatorname{ess\,sup}_{\nu(dy)} \pi_{Y=y}(A_{i_0}) - 1\right)\, \|f^*\|_{L^\infty(\mu)} \notag \\
&\geq (2\xi-1)\, \|f^*\|_{L^\infty(\mu)}.
\end{align*}
% Under the assumption, (iii) of Theorem \ref{corollary2} holds and \(T\) is surjective by this theorem.
\end{proof}
Suppose that \eqref{egtoutf} holds. By Remark \ref{munudiscrete}, if $\nu$ is discrete then $\mu$ is discrete, and by Corollary \ref{Lebesguedensity}, the pair $(X,Y)$, when $\mathcal{Y}=\mathbb{R}$, cannot admit a density with respect to the Lebesgue measure on $\mathbb{R}^2$. The following corollary of Proposition \ref{kkk} shows that these conditions are no longer necessary for \eqref{egtoutf2} to hold.
\begin{cor2}\label{cor2}
\begin{itemize}
  \item[(i)] There exists a probability measure $\pi$ on $\mathbb{R}^2$ admitting a density with respect to the Lebesgue measure such that \eqref{egtoutf2} holds.
  \item[(ii)] There exists a probability measure $\pi$ on $\mathbb{R}^2$ such that its first marginal $\mu$ is absolutely continuous with respect to the Lebesgue measure, its second marginal $\nu$ is discrete, and \eqref{egtoutf2} holds.
\end{itemize}
\end{cor2}
\begin{proof}
Let $\nu$ be a probability measure on $\mathbb{R}_+$ such that for every interval $I \subset \mathbb{R}_+$ of positive length, we have $\nu(I) > 0$. Let $\pi$ be a probability measure on $\mathbb{R}^2$ of second marginal $\nu$, such that $\nu(dy)$-a.e.,
\[
\pi_{Y = y}(dx)
\;=\;
\frac{\mathbf{1}_{A_y}(x)}{\lambda(A_y)} \, dx,
\]
with
$A_y \;=\; \bigl[\{y\} - 2^{-\lfloor y \rfloor},\;\{y\} + 2^{-\lfloor y \rfloor}\bigr] \cap [0,1],$
where $\{y\}$ denotes the fractional part of $y$ and $\lfloor y \rfloor$ denotes the floor of $y$.
Let $\delta\in(\frac{1}{2},1)$. We are going to check that $
\nu\bigl(\{\,y \in \mathbb{R}_+:\,\pi_{Y=y}(A)>\delta\}\bigr)>0
$ for all $A\in \mathcal{B}(\R)$ such that $\mu(A)>0$. By Proposition \ref{kkk}, this shows that the operator $T$ associated with $\pi$ is surjective.  
We conclude by taking
\begin{itemize}
    \item $\nu$ equivalent to $\mathbf{1}_{\mathbb{R}_+}(x) \, \lambda(dx)$, for (i),
    \item $\nu = \sum_{m \in \mathbb{N}} 2^{-m} \, \delta_{q_m}$, where $(q_m)_{m \in \mathbb{N}}$ is an enumeration of $\mathbb{Q}_+$, for (ii).
\end{itemize}
Note that $\mu$ is absolutely continuous with respect to $\mathbf{1}_{[0,1]}\,\lambda$. Let $A\in\mathcal{B}(\mathbb{R})$ be such that $\mu(A)>0$. Without loss of generality, we may assume that \(A \in \mathcal{B}([0,1])\), since \(\operatorname{supp}(\mu) \subset [0,1]\). By the Lebesgue differentiation theorem, we have
\[
\mathbf{1}_A(x)\, \lambda(dx) \text{-a.e., } \quad 
\frac{\lambda(A \cap [x-\varepsilon, x])}{\lambda([x-\varepsilon, x])}
\underset{\varepsilon \to 0^+}{\longrightarrow} 1,\quad
\frac{\lambda(A \cap [x, x+\varepsilon))}{\lambda([x, x+\varepsilon))}
\underset{\varepsilon \to 0^+}{\longrightarrow} 1.
\]
Hence, there exist $x\in(0,1)$ and $\varepsilon>0$ such that $[x-\varepsilon,x+\varepsilon]\subset[0,1]$ and
\begin{align*}
\forall \varepsilon' \in (0, \varepsilon), \quad
\frac{\lambda\bigl(A\cap[x-\varepsilon',x]\bigr)}{\varepsilon'}
>\delta,
\quad
\frac{\lambda\bigl(A\cap[x,x+\varepsilon']\bigr)}{\varepsilon'}
>\delta.
\end{align*}
In particular, it comes that
\begin{equation}\label{eq:interval-density}
\text{for every interval } I \text{ of positive length, such that } I \subset [x-\varepsilon, x+\varepsilon] \text{ and } x \in I, \text{ we have } \frac{\lambda\bigl(A\cap I\bigr)}{\lambda(I)} > \delta. 
\end{equation}
Let \(n\in\mathbb{N}\) and \(y\in[n,n+1)\). 
We observe that
\[
y \in \left[ n + x - \min\left(2^{-n},\, \varepsilon - 2^{-n} \right),\ 
             n + x + \min\left(2^{-n},\, \varepsilon - 2^{-n} \right) \right]
\quad \text{and} \quad 2^{-n} \leq \varepsilon
\]
implies that 
\[
A_y \subset [x - \varepsilon,\, x + \varepsilon] 
\quad \text{and} \quad x \in A_y.
\]
Thus for $n\geq \lceil-\log_2\varepsilon\rceil +1$, we have $2^{-n}\le \frac{e^{-\lceil-\log_2\varepsilon\rceil\log 2}}2\le \frac{\varepsilon}2$, so that 
\[
[n+x-2^{-n},n+x+2^{-n}] \subset \{\,y\in[n,n+1):A_y\subset[x-\varepsilon,x+\varepsilon] \text{ and } x\in A_y \}.
\]
Therefore, by hypothesis on $\nu$, we have
$
\nu\left(\left\{\, y \in [n, n+1) : A_y \subset [x - \varepsilon, x + \varepsilon]\text{ and } x\in A_y  \,\right\}\right)> 0.
$
By \eqref{eq:interval-density}, applied to $I=A_y$, it follows that
\begin{align*}
\nu\left(\left\{\,y \in \mathbb{R}_+:\,\pi_{Y=y}(A) > \delta\,\right\}\right)
&=
\nu\left(\left\{\,y \in \mathbb{R}_+:\,\frac{\lambda(A \cap A_y)}{\lambda(A_y)} > \delta\,\right\}\right) \\
&\ge
\nu\left(\left\{\,y \in \mathbb{R}_+:\, A_y \subset [x - \varepsilon, x + \varepsilon] \text{ and } x\in A_y \right\}\right)
\;>\; 0.
\end{align*}
\end{proof}
Theorem \ref{mainsuffisant} shows that if there exists a Borel set $\tilde A \subset \{ x \in \mathbb{R}: \mu(\{x \}) =0 \}$ such that $\mu(\tilde A) > 0$ and $
\nu(dy)\text{-a.e., }\pi_{Y=y}(\cdot \cap \tilde A) \ll \mu,
$
then \eqref{egtoutf} does not hold. In fact, if we further assume uniform integrability of the densities of $\pi_{Y=y}(\cdot \cap \tilde A)$, then the operator $T$ is not even surjective, as shown in the following corollary of Theorem \ref{corollary2}.
\begin{prop2}
If there exists a Borel subset $\tilde A \subset \{ x \in \mathbb{R}: \mu(\{x \}) =0 \}$ such that $\mu(\tilde A) > 0$ and $\nu(dy)$-a.e., $\mathbf{1}_{\tilde A}(x)\,\pi_{Y=y}(dx)=f_y(x)\,\mu(dx)$ with the family $(f_y)_{y\in\R}$ uniformly integrable with respect to \( \mu \), then \eqref{egtoutf2} does not hold.
\end{prop2}
\begin{proof}
Let $\delta>0$. By uniform integrability, there exists $\varepsilon>0$ such that for every Borel set $A\subset \tilde A$ with $\mu(A)<\varepsilon$, we have $\nu(dy)\text{-a.e., }
\pi_{Y=y}(A)<\delta.
$
There exists $z\in\mathbb R$ such that
$
0<\mu\bigl((-\infty,z]\cap \tilde A \bigr)<\varepsilon.
$
We set $A=(-\infty,z]\cap \tilde A$. We thus have $0<\mu(A)<\varepsilon$, so that $\mu(A)>0$ and $\nu(dy)$\text{-a.e., }
$
\pi_{Y=y}(A)<\delta.
$ Hence, $\|T^*\mathbf{1}_{A}\|_{L^\infty(\nu)}<\delta=\delta\|\mathbf{1}_{A}\|_{L^\infty(\mu)}$. Since $\delta$ is arbitrarily small, condition (ii) in  Theorem~\ref{corollary2} does not hold and, according to this theorem, $T$ is not surjective.
\end{proof}
\begin{proof}[Proof of Proposition \ref{corball}\\]
The implication (i) $\implies$ (ii) is clear. To prove the (ii) \,$\Longleftrightarrow\,\,  $(iii) we are going to check that both assertions are equivalent to  (iv) $\forall f^*\in L^\infty(\mu)$,\,
    $
    \|T^* f^*\|_{L^\infty(\nu)} = \|f^*\|_{L^\infty(\mu)}.$ The implication (ii) $\implies$ (iv) follows from Proposition \ref{kkk} combined with the non-expansiveness of $T^*$. Let us show that (iv) $\implies$ (ii).  
Let $A \in \mathcal{B}(\mathbb{R})$ be such that $\mu(A) > 0$, and consider $f^* = \mathbf{1}_A$.  
Since
$
\left\| T^*f^* \right\|_{L^{\infty}(\nu)} = \left\| f^* \right\|_{L^{\infty}(\mu)} = 1,
$
we conclude that
$
\operatorname{ess\,sup}_{\nu(dy)} \pi_{Y=y}(A) = 1.$ We introduce assertion
(iv) : $\forall f^*\in L^\infty(\mu)$,\,
    $
\|T^* f^*\|_{L^\infty(\nu)} = \|f^*\|_{L^\infty(\mu)}.$  Let us show that (iv) $\implies$ (iii). Theorem 4.13 of \cite{Rudin1991} gives $\mathring{B}_{L^1(\mu)}\subset T(\mathring{B}_{L^1(\nu)}) $ and since $T$ is non-expansive, it comes that $T(\mathring{B}_{L^1(\nu)}) = \mathring{B}_{L^1(\mu)}$. Finally, let us show that (iii) $\implies$ (iv).  
We recall (see  \cite[Theorem 4.3(b)]{Rudin1991}) the general identity in a Banach space $U$:
\[
\forall u^* \in U^*, \,\|u^*\|_{U^*} =
\sup \left\{ |\langle u^*,u  \rangle| : u\in U\mbox{ with }\|u\|_{U}\le 1\right\} 
=
\sup \left\{ |\langle u^*, u \rangle| : u\in U\mbox{ with }\|u\|_{U}< 1\right\}.
\]
Let $f^* \in L^\infty(\mu)$. We then have, using $T(\mathring{B}_{L^1(\nu)} ) = \mathring{B}_{L^1(\mu)}$ for the third equality,
\begin{align*}
\|T^* f^*\|_{L^\infty(\nu)}
&= \sup \left\{ |(T^* f^*,g )| : g \in \mathring{B}_{L^1(\nu)}  \right\} \notag \\
&= \sup \left\{ |\langle f^*, T g  \rangle| : g \in \mathring{B}_{L^1(\nu)}  \right\} \notag \\
&= \sup \left\{ | \langle f^*, f\rangle| : f \in  \mathring{B}_{L^1(\mu)}  \right\} \notag \\
&= \|f^*\|_{L^\infty(\mu)}.
\end{align*}
\end{proof}

\section{Conclusion}

% We have shown that, in practical financial situations, equation~\eqref{egtoutf} does not hold.

% Indeed, the joint distribution of \( (X, Y) \) typically admits a density with respect to the Lebesgue measure, so that \eqref{egtoutf} fails by Corollary~\ref{Lebesguedensity}.

% One can also observe that if the joint law of the pair \( (X, Y) \) is regular enough to satisfy the assumptions of Theorem~\ref{mainsuffisant}—which is always the case in practice—then \eqref{egtoutf} requires that \( \nu(D) > 0 \); that is, there must exist values of \( Y \) for which \( X \) is known. This excludes most path-dependent models, where \( (Y_t)_{t \geq 0} \) typically consists of weighted averages of the asset price, or of past returns. The asset price cannot be recovered if only the average value is known.

% It should nevertheless be kept in mind that this is an idealized formulation of the calibration problem. We require perfect calibration over a continuum of strikes and maturities, whereas in practice, only a finite number of strikes and maturities are observed.

We have essentially proved that under very mild assumptions on the joint distribution of $(X,Y)$, for $(\mathcal{R}_+)$ to hold, for all possible values $x$ taken by $X$, there must exist values $y$ attainable by $Y$ such that, knowing that $Y=y$, one has $X=x$. In particular, $(\mathcal{R}_+)$ does not hold when $\mathcal{Y}=\R^d$ and $(X,Y)$ has a density with respect to the Lebesgue measure on $\R^{1+d}$. Going back to the financial problem \eqref{eq:purePDVsmilecalib_discrete}, our results essentially mean that for \eqref{eq:purePDVsmilecalib_discrete} to hold (for some $\sigma^2(kh,\cdot)$) for any arbitrage-free local volatilities $\sigma_\text{loc}$, one must be able to reconstruct the current value of  the asset price from the values of the path-dependent variables, which is typically not the case---in particular when the current asset price and the path-dependent variables jointly admit a density with respect to the Lebesgue measure. However, our results do not mean that ``pure'' PDV models cannot be well calibrated to market smiles. Indeed, the formulation \eqref{eq:purePDVsmilecalib_discrete} of the smile calibration problem is some mathematical idealization of this problem: it assumes that we are given a full surface of implied volatilities indexed by a continuum of maturities and strikes, while in reality the market only provides bid and ask prices for a finite number of maturities and strikes. In fact, Gazzani and Guyon \cite{gazzani-guyon} have shown that PDV models can very accurately fit market smiles over a large range of maturities, in the absence of a leverage function or a deterministic input volatility term-structure.
\bigskip

\noindent{\bf Acknowledgement.} We thank Anthony Quas (University of Victoria) for suggesting the use of Lebesgue's differentiation theorem in the proofs of Lemma \ref{mainlemma} and Corollary \ref{cor2}.


\begin{thebibliography}{99}

\bibitem{AbergelTachet}
Abergel, F., Tachet, R.: \emph{On a Nonlinear Partial Integro-Differential Equation}, Discrete and Continuous Dynamical Systems, 27(3):907--917, 2010.

\bibitem{AndreBoumezoudJourdain2023}
Andrès, H., Boumezoued, A., Jourdain, B.: \emph{Implied volatility (also) is path-dependent}, preprint arXiv:2312.15950, 2023.

\bibitem{Bogachev2007v2}
Bogachev, I.: \emph{Measure Theory, Volume II}, Springer-Verlag, Berlin Heidelberg, 2007.


\bibitem{Bouchaud2021}
Bouchaud, J.-P.: \emph{Radical complexity}, Entropy 23:1676, 2021.


\bibitem{ChicheporticheBouchaud2014}
Chicheportiche, R., Bouchaud, J.-P.: \emph{The fine-structure of volatility feedback I: Multi-scale self-reflexivity}, Physica A 410:174–195, 2014.


\bibitem{DjeteNonRegular}
Djete, M.F.: \emph{Non–regular McKean–Vlasov Equations and Calibration Problem in Local Stochastic Volatility Models}, to appear in Mathematics of Operations Research, 2024.


\bibitem{Dudley2002}
Dudley, R.: \emph{Real Analysis and Probability}, Cambridge Studies in Advanced
Mathematics, vol.~74, Cambridge University Press, 2002.


\bibitem{Dupire1994}
Dupire B.: \emph{Pricing with a Smile}, Risk Magazine, 7(1):18--20, 1994.


\bibitem{FJWZ}
Friz, P.K., Jourdain, B., Wagenhofer, T., Zhou, A.: \emph{On the Weak Error for Local Stochastic Volatility Models}, preprint arXiv:2506.10817, 2025.



\bibitem{gazzani-guyon}Gazzani, G., Guyon, J.: \emph{Pricing and calibration in the 4-factor path-dependent volatility model}, Quantitative Finance 25(3):471--489, 2025.


\bibitem{GuyonLekeufack2023}
Guyon, J., Lekeufack, J.: \emph{Volatility is (mostly) path-dependent}, Quantitative Finance 23(9):1221–1258, 2023.



\bibitem{GuyonLabordere}
Guyon, J., Henry-Labordère, P.: \emph{Being particular about calibration}, Risk, January 2012.


\bibitem{Gyongy1986}
Gyöngy, I.: \emph{Mimicking the one-dimensional marginal distributions of processes having an Itô differential}, Probability Theory and Related Fields, 71(4):501--516, 1986.





\bibitem{Jones2001} 
Jones, F.: \emph{Lebesgue Integration on Euclidean Space}, Revised edition, Jones \& Bartlett Learning, 2001.


\bibitem{JourdainZhou2002}
Jourdain, B., Zhou, A.: \emph{Fake Brownian motion and calibration of a Regime Switching Local Volatility model}, Mathematical Finance 30(2):501–546, 2020.









\bibitem{LackerShkolnikovZhang}
Lacker, D., Shkolnikov, M., Zhang, J.: \emph{Inverting the Markovian Projection, with an Application to Local Stochastic Volatility Models}, Annals of Applied Probability 48(5):2189--2211, 2024. 




\bibitem{Renyi1958}
Rényi, A.: \emph{On mixing sequences of sets}, Acta Math. Acad. Sci. Hung. 9:215--228, 1958.


\bibitem{Rudin1991}
Rudin, W.: \emph{Functional Analysis}, 2nd ed., McGraw-Hill, 1991.

\bibitem{Stein2011}
Stein, E., Shakarchi, R.: \emph{Functional Analysis: Introduction to Further Topics in Analysis}, Princeton University Press, 2011.

\bibitem{Zumbach2009}
Zumbach, G.: \emph{Time reversal invariance in finance}, Quantitative Finance 9(5):505–515, 2009.

\bibitem{Zumbach2010}
Zumbach, G.: \emph{Volatility conditional on price trends}, Quantitative Finance 10(4):431–442, 2010.



\end{thebibliography}
\end{document}